\documentclass[reqno,11pt]{article}
\usepackage{amsthm}
\usepackage{e-jcCHANGE}

\usepackage{amsmath,amssymb,latexsym, xcolor, graphicx}

\def\pr{\text{ P\/}}
\def\ex{\text{E\/}}

\def\eps{\varepsilon}

\def\part{\partial}

\newtheorem{Theorem}{Theorem}[section]
\newtheorem{Lemma}[Theorem]{Lemma}

\newtheorem{Corollary}[Theorem]{Corollary}

\theoremstyle{remark}

\newtheorem*{Remark}{\bf Remark}
\numberwithin{equation}{section}

\title{\bf Birth of a giant $(k_1,k_2)$-core in the random digraph}

\author{Boris Pittel\\
\small Department of Mathematics\\[-0.8ex]
\small The Ohio State University\\[-0.8ex] 
\small Columbus, Ohio, U.S.A.\\
\small\tt bgp@math.osu.edu\\
\and
Daniel Poole \\
\small Department of Mathematics\\[-0.8ex]
\small The Ohio State University\\[-0.8ex]
\small Columbus, Ohio, U.S.A.\\
\small\tt poole@math.osu.edu
}

\begin{document}

\maketitle 

\begin{abstract} The $(k_1,k_2)$-core of a digraph is the largest sub-digraph with minimum
in-degree and minimum out-degree at least $k_1$ and $k_2$ respectively. For $\max\{k_1, k_2\} \geq 2$, we establish
existence of the threshold edge-density $c^*=c^*(k_1,k_2)$, such that the random digraph $D(n,m)$, on
the vertex set $[n]$ with $m$ edges, asymptotically almost surely has a giant $(k_1,k_2)$-core
if $m/n> c^*$, and has no $(k_1,k_2)$-core if $m/n<c^*$. Specifically, denoting  $\pr(\text{Poisson}(z)\ge k)$ by $p_k(z)$, we prove that $c^*=\min\limits_{z_1,z_2}\max\left\{\tfrac{z_1}{p_{k_1}(z_1)p_{k_2-1}(z_2)};
\tfrac{z_2}{p_{k_1-1}(z_1)p_{k_2}(z_2)}\right\}$.

\end{abstract}


\maketitle

\section{Main result and some prehistory} 

Let the fixed non-negative integers $k_1$, $k_2$ be such that $\max \{k_1, k_2\} \geq 2$. A $(k_1, k_2)$-core of a directed graph (digraph) $D=(V,E)$ on vertex set $V$ and set $E$ of directed edges is a maximal subdigraph with minimum in-degree and minimum out-degree at least $k_1$ and $k_2$ respectively. If a digraph does not have such a subdigraph, we say that the $(k_1, k_2)$-core is empty. 

We consider the random $D(n,m)$, the digraph chosen uniformly at random from all $\binom{(n)_2}{m}$ digraphs on $[n]$ with $m$ directed edges (arcs). As customary, for some $m=m(n)$, we say that some property holds {\it with high probability}, denoted {\it w.h.p.}, if the probability that $D(n,m)$ has this property tends to 1 as $n$ tends to infinity. We determine the sharp threshold for the existence of the $(k_1, k_2)$-core in $D(n,m)$. First some notations. Given $\text{Poi}(z)$, Poisson distributed random variable with parameter $z$, let $p_j(z):=\pr(\text{Poi}(z)\ge j)$. Introduce 
\[
c^*=c^*(k_1,k_2):=
\min\limits_{z_1,z_2>0}\max\left\{\frac{z_1}{p_{k_1}(z_1)p_{k_2-1}(z_2)};
\frac{z_2}{p_{k_1-1}(z_1)p_{k_2}(z_2)}\right\}.
\]
By symmetry, $c^*(k_1, k_2) = c^*(k_2, k_1)$. 
\begin{Theorem}\label{broadthm} $c^*$ is well defined, $c^*(0, k_2) = c^*(1, k_2) = \min \frac{z_2}{p_{k_2}(z_2)}$, and 
\begin{itemize}
\item for $c<c^* $, w.h.p. the $(k_1, k_2)$-core of $D(n, m=[c n])$ is empty;
\item for $c>c^* $, w.h.p. the $(k_1, k_2)$-core of $D(n, m=[c n])$ is not empty; in fact, there is some $\beta(c) = \beta(k_1, k_2, c)$ such that the $(k_1, k_2)$-core has $\beta n + O_p(n^{1/2} \log n)$ vertices. 
\end{itemize}
\end{Theorem}
For $c\downarrow c^*$, the in-degree and the out-degree of a generic vertex in the 
$(k_1,k_2)$-core are asymptotically independent, and distributed as $\text{Poi}(z_1^*)$ and $\text{Poi}(z_2^*)$, conditioned on $\{\text{Poi}(z_1^*)\ge k_1\}$ and $\{\text{Poi}(z_2^*)\ge k_2\}$ respectively. Furthermore, $\beta(c^*+)=p_{k_1}(z_i^*)\,p_{k_2}(z_o^*)$.

\begin{Remark}
The definition of $c^*(k_1, k_2)$ as the minimum of the maximum of two explicitly defined functions allows the interested reader to numerically determine $c^*$ for moderate sized $k_1$ and $k_2$. For instance, in the following table are numerical approximations for $c^*$ for $k_1, k_2 \leq 4$. 
\begin{center}
\begin{tabular}{|c|c|c|c|c|}
\hline
$k_1 \backslash k_2$ & 1 & 2 & 3 & 4 \\
\hline
$1$  &  n.a. &  3.351 & 5.148  &  6.799 \\
\hline
$2$  &  3.351 &  3.817 & 5.235  &  6.820 \\
\hline
$3$  &  5.148  &  5.235 &  5.768  &  6.971 \\
\hline
$4$  & 6.799  &   6.820 &  6.971 &  7.526  \\
\hline
\end{tabular}
\end{center}
\end{Remark}

\begin{Remark}
Though not immediately obvious from the definition of $c^*(k_1, k_2)$, we have that for $k \geq 2,$
\begin{equation*}
c^*(0, k) = c^*(1, k) < c^*(2, k) < \ldots < c^*(k,k).
\end{equation*}
A close look at the formula for $c^*(k_1, k_2)$ show that as $k \to \infty$, 
\begin{equation*}
c^*(k, k) = k + \sqrt{k \log \tfrac{k e^2}{2\pi}} - 1 + O\left( \sqrt{ k^{-1} \log k} \right),
\end{equation*}
and
\begin{equation*}
c^*(0, k) = k + \sqrt{k \log \tfrac{k}{2\pi}} - 1 + O\left( \sqrt{ k^{-1}/ \log k} \right),
\end{equation*}
so that $c^*(k,k)-c^*(0,k) \sim \sqrt{k / \log k}$.

\end{Remark}

In Figure \ref{fig: corepicture}, we produce a randomly sampled digraph with 50 vertices and 170 arcs, which corresponds to an arc density $c=3.4$ slightly above $c^*(1, 2) \approx 3.351$. For the empirical probability that a $(1,2)$-core exists, even 100 vertices does not give the small probabilities that we want. For 10,000 trials of $n=100, m=300$, the fraction of digraphs with a nontrivial $(1,2)$-core is roughly 27\%; for 10,000 trials of $n=100, m=350$, the fraction of digraphs with a nontrivial $(1,2)$-core is roughly 95\%. If we jump up an order, these probabilities get closer to what we want them to be. For 5,000 trials of $n=1000, m=3000$, zero of these digraphs had a $(1,2)$-core and for 5,000 trials of $n=1000, m=3500$, all but 17 had a $(1,2)$-core. Further, 5,000 trials of $n=1000$ and $m=3100, 3200, 3300, 3400$, the fraction of such digraphs with a $(1,2)$-core was 0.2\%, 9\%, 50\%, and 91\%, respectively. 

\begin{figure}[!h]
  \begin{center}
\scalebox{1}{\includegraphics{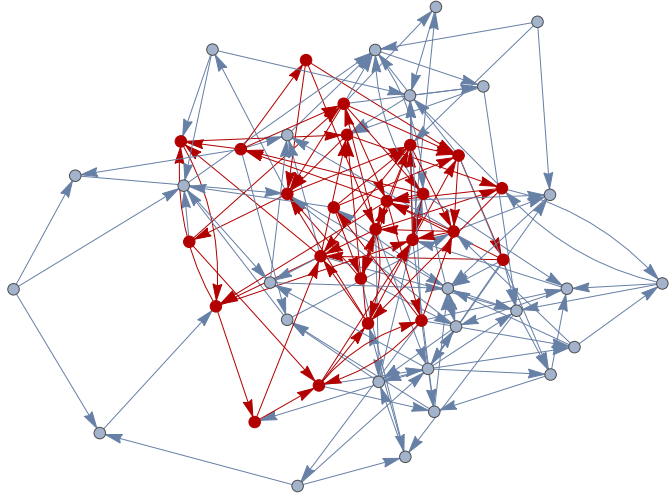}}
  \end{center}
  \caption{\label{fig: corepicture} (1,2)-core for $c=3.4$, $n=50$}
\end{figure}

\subsection{Connections to prior work}

The $k$-core has been well studied in the random undirected graph $G(n,m)$. In the pioneering paper \cite{Bol2}  Bollob\'{a}s defined the $k$-core of a graph to be the maximal subgraph with minimum degree at least $k$, and proved that for $8 \leq k+3 \leq c/n, c \geq 67$ with high probability (w.h.p.) the $k$-core of $G(n,m =cn/2)$ is non-empty and $k$-connected.  
 This breakthrough result opened a new area of analysis of the Erd\H os-R\'enyi random graph
process in the {\it postcritical\/} phase, i.e. beyond formation of a giant connected component, but long before the random graph becomes connected.
Later, T. \L uczak \cite{Luc1} proved that, for $k \geq 3$, w.h.p. the $k$-core of $G(n,m)$, when present, has at least $0.0002n$ vertices. Pittel, Spencer and Wormald \cite{PitSpeWor} introduced a randomized deletion algorithm which terminates with the $k$-core, and by analyzing the likely realization of this algorithm, for $k \geq 3$ they established the sharp threshold $c^*(k)$ for w.h.p. existence of the $k$-core, proving  that  $c^*(k)=\min\frac{z}{p_{k-1}(z)}$. (Thus $c^*(0,k_2)=c^*(1, k_2)=c^*(k_2+1)$.)  It was also proved in \cite{PitSpeWor} that the transition window for the edge density $c$ has width of order $O(n^{-\delta})$ for every $\delta<1/2$. 
More recently, Janson and M. \L uczak \cite{JanLuc} proved that the window width is of order
$n^{-1/2}$ exactly, and established a normal law for the size of the $k$-core in the supercritical phase, and a ``non-normal'' law at the threshold.

The interested reader will find in \cite{PitSpeWor} an informal 
explanation of the formula for $c^*(k)$ based on a deletion algorithm that at each round discards
all vertices of degree below $k$, and on the fact that, locally, the neighborhood of a generic vertex
is asymptotic to the first  few generations of the Galton-Watson branching process 
with Poi$(c)$-distributed immediate progeny. Later Molloy
\cite{Mol} and Riordan \cite{Rio} found proofs of the threshold $c^*(k)$, both based on this approach, with \cite{Mol} covering a general case of hypergraphs. The argument in \cite{Rio} uses a powerful (local) coupling of the graph $G(n, m)$ to the Galton-Watson process. There is a similar,
considerably more ``hand-waving'', explanation of how the parameters $z_1$ and $z_2$ enter
the stage, but it stops well short of a formula for $c^*(k_1,k_2)$. And  finding a
satisfying formal proof 
as conceptually transparent, and inherently simple, as our proof below seems a daunting task. To be sure, like
\cite{PitSpeWor}, and Aronson, Frieze, Pittel \cite{AroFriPit}, our argument will also be based on analysis of a randomized deletion algorithm, but with steps being much less radical:  each time
a {\it single\/}, uniformly random, ``light'' vertex is  deleted together with all edges
incident to it.  For this slowed-down deletion algorithm we can use the deterministic ODE system
as a potential approximation of the random work of the algorithm, a powerful approach
to random graphs championed by Wormald for many years.

This paper is closely related to (1) Pittel and Poole \cite{PitPoo}, in which we proved a normal, joint,
law for the size and the number of edges in the giant strong component of $D(n,m=[cn])$, $c>1$;
(2) two mutually complementary papers, P\'erez-Gim\'enez and  Wormald \cite{PerWor}, and Pittel \cite{Pit1} on asymptotic counting of strongly connected digraphs; (3) T. \L uczak and Seierstad \cite{LucSei} on the size of the strong giant for $c-1\gg n^{-1/3}$. Our earlier sources
of inspiration are Karp \cite{Kar1} and T. \L uczak \cite{Luc2} both on $c=1$ as the sharp threshold for
birth of the strong giant, and in particular, Karp's asymptotic formula $n\theta(c)^2$ ($\theta=1-e^{-c\theta}$) for the likely size of the giant for $c>1$.

\subsection{Outline for the proof}

First we introduce a randomized deletion algorithm that delivers  the, possibly empty, $(k_1, k_2)$-core of a given digraph.

To analyze its work on $D(n,m)$, in Section 2.1 we introduce a {\it directed\/} version of a random sequence model, originally invented by Chv\'atal \cite{Chv} for analysis of the $3$-colorability of the sparse
Erd\H os-R\'enyi random graph $G(n,m)$. The Chv\'atal model was instrumental in
\cite{AroFriPit} for a sharp analysis of the Karp-Sipser greedy matching algorithm in $G(n,m)$. 
The directed version enables us to view the deletion algorithm as a Markov process on
the set of the $(k_1+1)(k_2+1)$-tuples, whose components are counts of vertices of various,
relevant, in/out-degrees, with an additional component recording  the current number of edges.

In Section 2.2, we derive the asymptotic formulas for the expected, conditional, change of the
current $(k_1+1)(k_2+1)$-tuple. The computations are necessarily technical, but are conceptually simple, as  our approach does not require knowledge of the conditional distribution of that change.   
We like to think of this part as exploring the probabilistic infrastructure of the deletion algorithm.
The resulting list of $(k_1+1)(k_2+1)$ identities does look intimidating, but it leads to a system of 
 identities for changes of a leading subset of just $6$ parameters, no matter how large $k_1$
and $k_2$ are. To be sure this subsystem is not exactly closed, as it contains the current number
of non-isolated light vertices, not expressible through those $6$ parameters. 

In Section 3 we use these expected changes as a motivation for putting forth a deterministic
ordinary differential equation (ODE) system, of $(k_1+1)(k_2+1)$ equations, anticipating that its solution will be a sharp approximation of the random realization of the deletion process. 

In Section 3.1 we identify two integrals, i.e. two functions of those $6$ leading parameters,
that remain constant along the ODE trajectories, in a promising harmony with the integrals
in Pittel and Poole \cite{PitPoo} paper on the distribution of the giant $(1,1)$-core in $D(n,m)$,
 \cite{PitSpeWor} ($k$-core in $G(n,m)$),  and \cite{AroFriPit} (maximum matching in $G(n,m)$).

In Section 3.2 we establish a condition (Lemma \ref{1,2}) sufficient for the ODE trajectory to
terminate at a finite time.

In Section 3.3 we prove existence of the threshold density $c^*=c^*(k_1,k_2)$ for the finite-time termination of the ODE trajectory with initial conditions close to those for $D(n,m=[cn])$.

In Section 3.4 we translate the conditions for $c^*$ into a ``variational principle'',
Lemma  \ref{c*,genk1,k2}.

Finally, in Section 4 we use a general purpose theorem, due to Wormald \cite{Wor1}, and
a probabilistic counterpart of the argument in Section 3.3, to show that $c^*$ is the edge
density threshold for a giant $(k_1,k_2)$-core in $D(n,m=[cn])$.

\section{Deletion algorithm}  We begin with description of a deletion algorithm that delivers the $(k_1,k_2)$-core of a given digraph $D$. We  call a vertex  of a digraph light  if either
its in-degree is at most $k_1-1$, or its out-degree is at most $k_2-1$.  Set $D_0=D$. Recursively, in a current digraph $D_t$  we choose uniformly at random 
a non-isolated light vertex and delete all the edges incident to this vertex. This procedure determines the next digraph $D_{t+1}$.
The terminal digraph is the (possibly empty) $(k_1,k_2)$-core of the digraph $D$ complemented by a set of isolated vertices.

\subsection{Deletion process as a Markov chain on finite tuples} Our task is to analyze the likely behavior of this algorithm applied to the random digraph $D(n,m)$. We notice upfront that, for $k:=\max\{k_1, k_2\} \geq 2$, with probability $1-O\left(n^{-k+3/2}\right)$ the $(k_1, k_2)$-core of $D(n, m=[cn])$ is either empty, or of size $0.9 \, \alpha (k, c) \, n$ at least, where
\begin{equation}\label{akc bound}
\alpha(k,c) := \min \left \{ \left(\frac{e^{k+1} \, c^k}{k^k} \right)^{\frac{1}{k-1}} , \frac{k}{c \, e} \right \}.
\end{equation}
This claim is a directed counterpart of a well-known result of \L uczak \cite{Luc1} for the $k$-core in the random graph $G(n, p=c/n), k \geq 3$. Similarly to the argument in that paper, the proof is based on the asymptotic estimate of the expected number of sub-digraphs with the number of edges at least $k$ times the number of vertices. Thus we may and will stop the deletion process once the total number of vertices with in-degree and out-degree exceeding $k_1$ and $k_2$ respectively, drops below, says, $0.8 \alpha(k, c)n$: if continued, w.h.p. the deletion process will end with the empty digraph.\\

Now down to brass tacks. To handle the considerable technical details, we introduce an auxiliary  random sequence model. This model is a directed version of the Chv\'atal random sequence model \cite{Chv}, which was already used by Aronson, Frieze and Pittel \cite{AroFriPit} for analysis of a vertex deletion process at the heart of the Karp-Sipser greedy matching algorithm \cite{KarSip}.

Given a sequence $\bold x=(x_1,\dots,x_{2m})$, $x_i\in [n]$, we define a {\it multi\/}-digraph
$D_{\bold x}$ with vertex set $[n]$ and (directed) edge set $\bigl[\{x_{2r-1},x_{2r}\}:\,1\le r\le m\bigr]$;
thus $e_{\bold x}(i,j)$, the number of directed edges $i\to j$, is $|\{r: x_{2r-1}=i,x_{2r}=j\}$. In essence,
$\bold x$ is a full, $m$-long record of throwing in $m$ directed edges into the initially empty digraph, one edge
at a time with loops and parallel edges allowed. The in-degree sequence 
$\boldsymbol\delta_{\bold x}$
and the out-degree sequence $\boldsymbol\Delta_{\bold x}$ of $D_{\bold x}$  are given by
\[
\delta_{\bold x}(i)=|\{r: x_{2r}=i\}|,\quad \Delta_{\bold x}(i)=|\{r:x_{2r-1}=i\}|,\quad 1\le i\le n,
\]
so that 
\[
\sum_{I\in [n]}\delta_{\bold x}(i)= \sum_{I\in [n]}\Delta_{\bold x}(i)=m.
\]
If $\bold x$ is distributed uniformly on the set $[n]^{2m}$ then $D_{\bold x}$ can serve as a ``surrogate''
for $D(n,m)$, meaning the following.  First of all, conditioned on the event ``$D_{\bold x}\text{ is simple}$'' (i.e. no loops or parallel edges), $D_{\bold x}$ is distributed as $D(n,m)$. Second, for $m/n\to c<\infty$,
\begin{equation*}
\pr(D_{\bold x}\text{ is simple})=\binom{(n)_2}{m}\cdot \frac{m!}{n^{2m}}\to e^{-c-c^2/2}>0.
\end{equation*}
Thus uniformly over all events  $A$ 
\begin{equation}\label{4}
\pr(D(n,m)\in A) =O\bigl(\!\!\pr(D_{\bold x}\in A)\bigr).
\end{equation}
Therefore to show that an event $A(n,m)$ is unlikely for $D(n,m)$ it suffices to prove that $A(n,m)$ is unlikely for  $D_{\bold x}$. And $D_{\bold x}$ is incomparably easier to deal with. We can view
the random sequence $\bold x$ as a full record of throwing  {\it alternately\/} $m$  ``out-balls'' and $m$ ``in-balls'' into $n$ boxes {\it one ball at a time\/} (starting with an out-ball) independently of one other. So $x_{2r-1}=i$, $x_{2r}=j$ means that  the $r$-th out-ball and the $r$-th in-ball went into the box $i$ and the box $j$ respectively, signifying birth of a directed edge $e_{\bold x}(i,j)$ from vertex $i$ to vertex $j$.  So the in-degree sequence $\boldsymbol\delta_{\bold x}=\{\delta_{\bold x}
(j)\}_{j\in [n]}$ (out-degree sequence $\boldsymbol\Delta_{\bold x}=\{\Delta_{\bold x}(i)\}_{i\in [n]}$ resp.) is the collection of 
``in-occupancy'' (``out-occupancy'') numbers for the $n$ boxes representing the vertices. 
In particular, $\boldsymbol\delta_{\bold x}$ and $\boldsymbol\Delta_{\bold x}$  are mutually independent, each distributed multinomially, with  $m$ trials and $n$ equally likely outcomes in each trial, a property crucial for analysis of the deletion algorithm. 

Let us describe one step of the deletion algorithm applied to a multi-digraph $D_{\bold x}$ in terms of the underlying sequence $\bold x$. This algorithm delivers a sequence
$\{\bold x(t)\}$ where $\bold x(0)=\bold x$, and each $\bold x(t)\in M:=([n]\cup \{\star\})^{2m}$, where for all $r$, $x_{2r-1}(t)=\star$ if and only if $x_{2r}(t)=\star$. The $\star$ pairs mark the 
locations $(2r-1,2r)$ in the original $\bold x(0)$ whose vertex occupants have been deleted after
$t$ steps. 

Recursively, at step $t+1$  we  (1)  select a vertex $i$ uniformly at random
among all non-isolated light vertices $j$, i.e. those with either $\delta_{\bold x}(j)<k_1$ or $\Delta_{\bold x}(j)<k_2$; (2) identify all the pairs $x_{2r-1},x_{2r}$ such that at least one
of the occupants $x_{2r-1}$, $x_{2r}$ is $i$, and replace each such pair with the symbol $\{\star,\star\}$, to get $\bold x(t+1)$. 

Given a generic $\bold x\in M$, define 
\[
\mathcal S_{\bold x}=\bigl(\{V_{a,b}(\bold x)\}, \{V_{a,\bullet}(\bold x)\}, \{V_{\bullet,b}(\bold
x)\}, V(\bold x), 
\mu(\bold x)\bigr);
\]
here $a<k_1$, $b<k_2$, $V_{a,b}(\bold x)$ is the set of all (doubly) light vertices, in $\bold x$, with in-degree $a$ and out-degree $b$; $V_{a,\bullet}(\bold x)$ is the set of all
(semi) light vertices of in-degree $a$ and out-degree $\ge k_2$; $V_{\bullet,b}(\bold x)$ is the set of all
(semi) light vertices of in-degree $\ge k_1$ and out-degree $b$;  $V(\bold x)$ is the set of all (doubly) heavy vertices, i.e. with in-degree $\ge k_1$ and out-degree $\ge k_2$; $\mu(\bold x)$ is the
number of  non-star pairs in $\bold x$. Everywhere below the subindeces (superscripts)  ``$a$''
(``$b$''  resp.) will also denote a generic value of light in-degree (light out-degree resp.).  Set 
\[
\bold s_{\bold x}= \bigl(\{v_{a,b}(\bold x)\}, \{v_{a,\bullet}(\bold x)\}, \{v_{\bullet,b}(\bold x)\}, v(\bold x), 
\mu(\bold x)\bigr),
\]
with the first four components being  cardinalities of the set components of $\mathcal S_{\bold x}$.
The sequence 
$\{\bold x(t)\}$ determines the sequence $\{\bold s(t)\} :=\{\bold s_{\bold x(t)}\}$. For brevity,
we will write $v_{a,b}(t),\dots,\mu(t)$ instead of $v_{a,b}(\bold x(t)),\dots, \mu(\bold x(t))$.
We want to show that $\{\bold s(t)\}$ is a Markov chain.\\

The following two claims are ``directed'' counterparts of Lemma 2 and Lemma 3 in \cite{AroFriPit}.
\begin{Lemma}\label{x(t)uniform} Given $\bold s$, let $M_{\bold s}=\{\bold x: \bold s_{\bold x}=\bold s\}$. Suppose that $\bold x(0)$ is distributed uniformly on  $M_{\bold s(0)}$. Then, for all 
$t\ge 0$, conditioned on $\bold s(0),\,\bold s(1),$ $\dots,\bold s(t)$, the sequence $\bold x(t)$ is distributed uniformly on $M_{\bold s(t)}$.
\end{Lemma}
\begin{proof} We prove this by induction on $t$. It is true for $t=0$.  Indeed $\bold x(0)=\bold x$
is distributed uniformly on $[n]^{2m}$. Therefore, conditioned on  $\bold s(0)$, $\bold x(0)$  is uniformly distributed on $M_{\bold s(0)}\subseteq [n]^{2m}$. Suppose the claim holds for
some $t\ge 0$. Let us prove the induction step. \\

{\bf (1)\/} First we show that for generic $\bold s$, $\bold s^\prime$,  each $\bold x^\prime\in M_{\bold s^\prime}$ arises by an admissible transition of the edge-deletion
algorithm from the same number $D(\bold s,\bold s^\prime)$ of $\bold x\in M_{\bold s}$. Suppose $\mu^\prime=\mu-k$, for some $k \geq 1$.

To select a generic $\bold x^\prime\in M_{\bold s^\prime}$ we (1) choose  a partition $\bold V^\prime$ of the vertex set
\begin{equation}\label{partition}
[n]=\bigcup_{a,b} V^\prime_{a,b}\,\,\bigcup_a V^\prime_{a,\bullet}\,\,\bigcup_b V^\prime_{\bullet,b}\,\,\bigcup V^\prime,
\end{equation}
with $|V^\prime_{a,b}|=v^\prime_{a,b}$, $|V^\prime_{a,\bullet}|=v^\prime_{a,\bullet}$, $|V^\prime_{\bullet, b}|=v^\prime_{\bullet,b}$, $|V^\prime|=v^\prime$; 
(2) select $\mu^\prime$  pairs among $m$ pairs $\{2r-1,2r\}$ and fill them with the vertex pairs $\{i,j\}$ such that the resulting in/out-degree sequence is compatible with the partition $\bold V^\prime$, putting the pairs $\{\star,\star\}$ into the remaining $m-\mu^\prime$ pairs $\{2r-1,2r\}$.

To undo the deletion step, we need to identify a vertex $i\in V^\prime_{0,0}$ and replace some $k$ pairs $\{\star,\star\}$ in $\bold x^\prime$ with
edges $\{u_\ell,v_\ell\}$, $1 \leq \ell \leq k$ chosen such that (1) for each $\ell$, at least one
of the vertices $u_{\ell}$, $v_{\ell}$ is $i$; (2) 
the in/out degrees for the resulting sequence $\bold x$ are compatible with the counts $v_{a,b}$, $v_{a,\bullet}$, $v_{\bullet,b}$, and $v$, and vertex $i$ is light.
 Clearly the number of ways to do this is completely
determined by the partition $\bold V^\prime$, and those counts $v_{a,b}$, $v_{a,\bullet}$, $v_{\bullet,b}$, and $v$. But then this number is a function $D(\bold s,\bold s^\prime)$, i.e.  it
depends only the blocks
cardinalities $v^\prime_{a,b}$, $v^\prime_{a,\bullet}$, $v^\prime_{\bullet, b}$ and $v^\prime$. Indeed a permutation on $[n]$, that transforms one such partition $\bold V^\prime$ of $[n]$ into another given
partition $\bold V^\prime_1$, induces a bijection between the two corresponding sets of the ways to undo the deletion step.\\

{\bf (2)\/} Next, if $\bold x^\prime\in M_{\bold s(t+1)}$, then the inductive assumption and the Markov property
of the process $\{\bold x(t)\}_{t\ge 0}$ implies---via conditioning on $\bold s(t)$---that
\begin{equation}\label{viacond}
\pr(\bold x(t+1)=\bold x^\prime\,|\,\{\bold s(\tau)\}_{\tau\le t})=\frac{1}{|M_{\bold s(t)}|}\sum_{\bold x\in M_{\bold s(t)}}
\pr(\bold x(t+1)=\bold x^\prime\boldsymbol |\,\bold x(t)).
\end{equation}
Now, the number of choices of a transition $Tr$ available for the deletion process applied
to $\bold x\in M_{\bold s(t)}$ is $L(\bold s(t))$, where $L(\bold s(t))$ is the total number of non-isolated light vertices, completely determined by $\bold s(t)$. Hence
\[
\pr(\bold x(t+1)=\bold x^\prime\,|\,\bold x(t))=\frac{1}{L(\bold s(t))}\sum_{Tr} 1_{\{\bold x^\prime\text{ arises
from }\bold x(t)\text{ via }Tr\}}.
\]
Using \eqref{viacond}, we obtain then
\begin{equation}\label{above}
\pr(\bold x(t+1)=\bold x^\prime\,|\,\{\bold s(\tau)\}_{\tau\le t})=\frac{D(\bold s(t),\bold s_{\bold x^\prime})}{|M_{\bold s(t)}|\,L(\bold s(t))}.
\end{equation}
This transition probability depends only on the current tuple $\bold s(t)$ and the next
tuple $\bold s(t+1)=\bold s_{\bold x^\prime}$, rather than on the full value of $\bold x(t+1)=\bold x^\prime$ in the set $M_{\bold s(t+1)}$. Thus
\[
\pr(\bold x(t+1)=\bold x^\prime\,|\,\{\bold s(\tau)\}_{\tau\le t+1})=\frac{1}{|M_{\bold s_{\bold x^\prime}}|},
\]
and the proof is complete.
\end{proof}

\begin{Lemma}\label{v(t)Markov} 
The random sequence of the tuples $\{\bold s(t)\}_{t\ge 0}$,
is a time-\linebreak homogeneous Markov chain with the transition probability 
\[
p(\bold s(t+1)=\bold s^\prime\,|\,\bold s(t)=\bold s):=\frac{D(\bold s,\bold s^\prime)
|M_{\bold s^\prime}|}{|M_{\bold s}|\,L(\bold s)},
\]
if $L(\bold s) \neq 0$. 
\end{Lemma}
\begin{proof} Follows immediately from \eqref{above}.
\end{proof}

\subsection{Expected one-step transitions} Let $\bold s$ be given. Suppose $\bold x$ is chosen uar from $M_{\bold s}$ and one step of the deletion algorithm is carried out, yielding $\bold x^\prime$. Let $\bold s^\prime=\bold s_{\bold x^\prime}$.  Our task is to estimate sharply $\ex[\bold s^\prime-\bold s\,|\,\bold s]$.\\

{\bf Step 1.\/}  First we need to determine the vertex degree distribution of the random
$\bold x\in M_{\bold s}$. To make formulas easier on the eye, we will use $\bold X=\{ X_j; j\in [n]\}$,
$\bold Y=\{Y_j; j\in [n]\}$ to denote the in/out degree sequences of the uniformly random $\bold x\in M_{\bold s}$, and continue to use $\delta$'s, $\Delta$'s  for generic values of individual
vertex in/out degrees $X_i$, $Y_i$. 
Since $\bold s$ contains
full information on counts of vertices with either light in-degree, or light out-degree, our focus
will be on vertices with either in-degree $\ge k_1$ or, not exclusively,  out-degree $\ge k_2$.
Recall that
\begin{align*}
&\bold s=\bigl(\{v_{a,b},v_{a,\bullet},v_{\bullet, b},v\}_{\{a<k_1,\,b <k_2\}},\mu\bigr),\\
&\sum_{a,b}v_{a,b} +\sum_a v_{a,\bullet} + \sum_ b v_{\bullet,b}+v=n.
\end{align*}
To generate the elements $\bold x$ of $M_{\bold s}$ we
\begin{align}
&\bullet\text{select }\bold V,\text{ a partition of }[n]\text{ into blocks }V_{a,b}, V_{a,\bullet}, V_{\bullet,b},V
\text{ of sizes }\notag\\
&\quad v_{a,b}, v_{a,\bullet} , v_{\bullet,b}, v, \text{ for all } a<k_1, \,b<k_2,\,\text{ in }\notag\\
&\qquad\qquad\qquad \binom{n}{\bold v}:=\frac{n!}{\prod_{a,b}v_{a,b}!\prod_a v_{a,\bullet}!\prod_b v_{\bullet,b}!\, v!}\,\,\text{ ways},\label{binv}\\
&\,\,\,\text{ and set }J=J(\bold V)=V\cup\,(\cup_b V_{\bullet,b}),\text{ and }I=I(\bold V)=V\cup\,(\cup_aV_{a,\bullet});\notag\\
&\bullet \text{select }\mu\text{ non-star pairs out of total }m\text{ pairs in }\binom{m}{\mu}\text{ ways},\notag\\
&\bullet\text{choose a matrix }\mathcal M=\{\mu_{i,j}\}_{i,j\in [n]}\text{ such that }\sum_{i,j}\mu_{i,j}=\mu,\label{summuij}\\
&\,\,\,\, \sum_{i\in [n]}\mu_{i,j}=a\,\, (\ge k_1 \text{ resp.}),\,\text{for } j\in (\cup_b V_{a,b})\,\cup V_{a,\bullet}\,\,\,(\text{for } j\in J \text{ resp.}),
\label{indeg}\\
&\,\,\,\, \sum_{j\in [n]}\mu_{i,j}=b\,\, (\ge k_2 \text{ resp.}),\,\text{for } i\in (\cup_a V_{a,b})\cup V_{\bullet,b}\,\,\,(\text{for } i\in I \text{ resp.}),
\label{outdeg}\\
&\bullet\text{for each }\{i,j\} \text{ select and fill }\mu_{i,j} \text{ non-star pairs with labels }i,j. \notag
\end{align}
For each realization of this $4$-step selection we obtain a distinct $\bold x\in M_{\bold s}$. Introduce 
\[
\delta_j=\sum_{i\in[n]}\mu_{i,j},\quad \Delta_i=\sum_{j\in [n]}\mu_{i,j},
\]
the in/out-degrees of the  resulting $\bold x$. According to \eqref{indeg}-\eqref{outdeg},
\begin{equation}\label{fullinout}
\begin{aligned}
&\delta_j=a\,\, (\ge k_1 \text{ resp.}),\,\text{ for } j\in (\cup_b V_{a,b})\,\cup V_{a,\bullet}\,\,\,(\text{for } j\in J \text{ resp.}),\\
&\Delta_i=b\,\, (\ge k_2 \text{ resp.}),\,\text{ for }i\in (\cup_a V_{a,b})
\cup V_{\bullet,b}\,\,\,(\text{for } i\in I \text{ resp.}).
\end{aligned}
\end{equation}
Thus $J$ ($I$ resp.) is the set of vertices of in-degree $\ge k_1$ (out-degree $\ge k_2$ resp.)
The number of ways to choose a matrix $\mathcal M$ with the in-degrees $\boldsymbol\delta$ and the out-degrees $\boldsymbol\Delta$ {\it and\/} to fill the $\mu$ vacant locations $\{\cdot,\cdot\}$ with $\mu_{i,j}$ pairs  $\{i,j\}$, ($i,j\in [n]$),
is
\begin{equation}\label{binprod} 
\begin{aligned}
\binom{\mu}{\boldsymbol\delta}\cdot\binom{\mu}{\boldsymbol\Delta}&=\frac{(\mu!)^2}{\prod_{a<k_1}(a!)^{v^a}
\prod_{b<k_2}(b!)^{v^b}}\\
&\,\,\,\,\,\,\times\prod\limits_{j\in J}\!\,1/\delta_j! \,\,\,\,\,\,\cdot\prod\limits_{i\in I}
\!\,1/\Delta_i!,
\end{aligned}
\end{equation}
where
\begin{equation}\label{va,vb}
\begin{aligned}
v^a&:=|\cup_ b V_{a,b}|+|V_{a,\bullet}|=\sum_{b<k_2}v_{a,b}+v_{a,\bullet},\\
 v^b&:=|\cup_a V_{a,b}|+|V_{\bullet,b}|=\sum_{a<k_1} v_{a,b}+v_{\bullet,b}
\end{aligned}
\end{equation}
are the total number of vertices with in-degree $a<k_1$ (with out-degree $b<k_2$ resp.). Notice that it is the second line expression in \eqref{binprod} that is not determined  by $\bold s$ alone. We know that  $\delta_j\ge k_1$ for $j\in J$, $\Delta_i\ge k_2$ for 
$i\in I$.  Also
\begin{equation}\label{sumsdelDel}
\begin{aligned}
\sum_{j\in J} \delta_j &=\mu - \mu_{i},\quad \mu_{i}:=\sum_a a v^a,\\
\sum_{i\in I}\Delta_i&=\mu-\mu_{o},\quad \mu_{o}:=\sum_b b v^b,
\end{aligned}
\end{equation}
and  
\begin{equation}\label{vi,vo}
|J|=v^{i}:=v+\sum_b v_{\bullet,b};\quad |I|=v^{o}:=v+\sum_a v_{a,\bullet}.
\end{equation}
Here $\mu_{i}$ ($\mu_{o}$ resp.) is the total in-degree (out-degree resp.) of vertices with maximum in-degree (out-degree resp.) below $k_1$ (below $k_2$ resp.); further  $v^{i}$ ($v^{o}$ resp.) is the total number of vertices with in-degree (out-degree resp.) at least $k_1$ (at least $k_2$ resp.). 

So, denoting $\hat{\boldsymbol\delta}=\{\delta_j: j\in J\}$, $\hat{\boldsymbol\Delta}=
\{\Delta_i: i\in I\}$,
\begin{equation}\label{g(s)prelim}
\begin{aligned}
g(\bold s)&=\binom{n}{\bold v}\binom{m}{\mu}\cdot\frac{(\mu!)^2}{\prod_{a<k_1}(a!)^{v^a}
\prod_{b<k_2}(b!)^{v^b}}\\
&\,\,\,\,\,\,\times\sum_{(\hat{\boldsymbol\delta},\hat{\boldsymbol\Delta})\text{ meet }\eqref{sumsdelDel},\atop \delta_j\ge k_1,\,\Delta_i\ge k_2}\,\,
\prod\limits_{j\in J}1/\delta_j!\,\,\prod\limits_{i\in I}1/\Delta_i! .
\end{aligned}
\end{equation}
For the sum to be non-zero, we need to have
\begin{equation}\label{edgedensity>}
\mu-\mu_{i}\ge k_1 |J|=k_1v^{i},\quad \mu-\mu_{o}\ge k_2 |I|=k_2v^{o}.
\end{equation}
Enter the generating functions! Introducing indeterminates $z_{i}$,$z_{o}$, the second line sum in \eqref{g(s)prelim} equals
\begin{equation}\label{[coeff]}
\begin{aligned}
&[z_{i}^{\mu-\mu_{i}}z_{o}^{\mu-\mu_{o}}]\sum_{\delta_j\ge k_1,\,j\in J
\atop \Delta_i\ge k_2,\,i\in I}\,\prod\limits_{j\in J}\,\,z_{i}^{\delta_j}/\delta_j!\,\,\prod\limits_{i\in I}\,\,z_{o}^{\Delta_i}/\Delta_i! \\
=&[z_{i}^{\mu-\mu_{i}}z_{o}^{\mu-\mu_{o}}]\, f_{k_1}(z_{i})^{v^{i}}
f_{k_2}(z_{o})^{v^{o}};\qquad (f_k(z):=\sum_{j\ge k}z^j/j!).
\end{aligned}
\end{equation}
Thus, combining \eqref{binv}, \eqref{binprod}, \eqref{g(s)prelim} and \eqref{[coeff]}, we  have proved
\begin{Lemma}\label{g(s)better}
\begin{align*}
g(\bold s)=&\binom{n}{\bold v}\binom{m}{\mu}\cdot\frac{(\mu!)^2}{\prod_{a<k_1}(a!)^{v^a}
\prod_{b<k_2}(b!)^{v^b}}\\
&\times [z_{i}^{\mu-\mu_{i}}z_{o}^{\mu-\mu_{o}}]\, f_{k_1}(z_{i})^{v^{i}}
f_{k_2}(z_{o})^{v^{o}},
\end{align*}
with ($v^a$, $v^b$), ($\mu_{i}$, $\mu_{o}$),  and ($v^{i}$, $v^{o}$) defined in \eqref{va,vb}, \eqref{sumsdelDel} and \eqref{vi,vo} respectively.
\end{Lemma}
To proceed, introduce $g(\bold V,\mu)$,  the total number of $\bold x\in M_{\bold s}$ with a fixed
partition $\bold V$ of $[n]$  into blocks $V_{a,b}, V_{a,\bullet}, V_{\bullet,b},V$
of sizes $v_{a,b}, v_{a,\bullet} , v_{\bullet,b}, v$, for all $a<k_1$, $b<k_2$. By symmetry, it follows from Lemma
\eqref{g(s)better} that 
\begin{equation}\label{g(V,mu)=}
\begin{aligned}
g(\bold V,\mu)=&\binom{m}{\mu}\cdot\frac{(\mu!)^2}{\prod_{a<k_1}(a!)^{v^a}
\prod_{b<k_2}(b!)^{v^b}}\\
&\times [z_{i}^{\mu-\mu_{i}}z_{o}^{\mu-\mu_{o}}]\, f_{k_1}(z_{i})^{v^{i}}
f_{k_2}(z_{o})^{v^{o}}.
\end{aligned}
\end{equation}
In other words, $(\bold V,\mu)$ assumes each of its values with the same probability $\binom{n}
{\bold v}^{-1}$. 
Introduce $g(\bold V,\mu;\hat{\boldsymbol\delta}, \hat{\boldsymbol\Delta})$, the total
number of sequences $\bold x\in M_{\bold s}$ with a fixed partition $\bold V$ and the in-degree (out-degree resp.) sequence for vertices in $J=J(\bold V)$ ($I=I(\bold V)$) equal to
$\hat{\boldsymbol\delta}$ ($\hat{\boldsymbol\Delta}$ resp.). The admissible
$\hat{\boldsymbol\delta}$, $\hat{\boldsymbol\Delta}$ must meet the conditions $\delta_j\ge k_1$ , $j\in J=J(\bold V)$, $\Delta_i\ge k_2$, $i\in I=I(\bold V)$ and \eqref{sumsdelDel}.
Arguing as in derivation of \eqref{g(s)prelim}, we have 
\begin{equation}\label{g(V,mu,hats)=}
\begin{aligned}
g(\bold V,\mu;\hat{\boldsymbol\delta}, \hat{\boldsymbol\Delta})&=\binom{m}{\mu}\cdot\frac{(\mu!)^2}{\prod_{a<k_1}(a!)^{v^a}
\prod_{b<k_2}(b!)^{v^b}}\\
&\quad\times\prod\limits_{j\in J}1/\delta_j!\,\,\prod\limits_{i\in I}1/\Delta_i! .
\end{aligned}
\end{equation}
Set 
$\hat{\bold X}=\{X_j; j\in J\}$, $\hat{\bold Y}=\{Y_i;i\in I\}$. Of course, the remaining components of $\bold X$ and of $\bold Y$ are uniquely determined by $\bold V$. From \eqref{g(V,mu)=}--\eqref{g(V,mu,hats)=} it follows that
\begin{equation}\label{dist(X,Y)}
\pr(\hat{\bold X}=\hat{\boldsymbol\delta},\,\hat{\bold Y}=\hat{\boldsymbol\Delta}\,|\,\bold V,\mu)=
\frac{\prod_j 1/\delta_j!}{[z_{i}^{\mu-\mu_{i}}] f_{k_1}(z_{i})^{v^{i}}}\cdot
\frac{\prod_i1/\Delta_i!}{
[z_{o}^{\mu-\mu_{o}}]
f_{k_2}(z_{o})^{v^{o}}},
\end{equation}
Thus, conditioned on $\bold V$, $\mu$, the vectors $\hat{\bold X}$ and $\hat{\bold Y}$ are {\it mutually\/} independent, each of the components of $\hat{\bold X}$ ($\hat{\bold Y}$ resp.)
having a common distribution, that of $X$ and of $Y$ respectively: for $\delta\ge k_1$, $\Delta\ge k_2$,
\begin{equation}\label{X,Y}
\begin{aligned}
\pr(X=\delta):=\!\pr(X=\delta\boldsymbol |\bold V,\mu)&=\frac{1}{\delta!}\,\frac{\bigl[\eta^{\mu-\mu_i-\delta}\bigr]
f_{k_1}(\eta)^{v^i-1}}{\bigl[\xi^{\mu-\mu_i}\bigr]f_{k_1}(\xi)^{v^i}},\\
\pr(Y=\Delta):=\!\pr(Y=\Delta\boldsymbol |\bold V,\mu)&=\frac{1}{\Delta!}\,\frac{\bigl[\eta^{\mu-\mu_o-\Delta}\bigr]
f_{k_2}(\eta)^{v^o-1}}{\bigl[\xi^{\mu-\mu_o}\bigr]f_{k_2}(\xi)^{v^o}}.
\end{aligned}
\end{equation}
There is a more tractable approximation for the distributions of $X$ and $Y$, applicable 
for a sufficiently large range of $\bold s$.\\

Fix $z_{i}>0$, $z_{o}>0$ and define two
{\it truncated\/} Poissons $Z_{i}$ and $Z_{o}$,
\begin{equation*}
\pr(Z_{i}=\delta)=\frac{z_{i}^{\delta}}{\delta! f_{k_1}(z_{i})},\,\,\delta\ge k_1,\quad
\pr(Z_{o}=\Delta)=\frac{z_{o}^{\Delta}}{\Delta! f_{k_2}(z_{o})},\,\,\Delta\ge k_2;
\end{equation*}
so $Z_{i,o}$ is $\text{Poi}(z_{i,o})$ conditioned on $\text{Poi}(z_{i,o})\ge k_{1,2}$. Introduce $\hat{\bold Z}_{i}=\{Z_{i}^{(j)},\,j\in J\}$, the $|J|$-long sequence of independent copies of $Z_{i}$, and $\hat{\bold Z}_{o}=\{Z_{o}^{(i)},\,i\in I\}$, the $|I|$-long sequence of independent copies of $Z_{o}$. Using $\hat{\bold Z}_{i}$ and $\hat{\bold Z}_{o}$,
we rewrite the equations \eqref{X,Y} as follows: for $\delta\ge k_1$, $\Delta\ge k_2$,
\begin{equation}\label{X,Y:Zi,Zo}
\begin{aligned}
\pr(X=\delta)&=\pr(Z_i=\delta)\,\frac{\pr\!\left(\sum_{j=1}^{v^i-1}Z_i^{(j)}=\mu-\mu_i-\delta
\right)}{\pr\!\left(\sum_{j=1}^{v^i}Z_i^{(j)}=\mu-\mu_i\right)},\\
\pr(Y=\Delta)&=\pr(Z_o=\Delta)\,\frac{\pr\!\left(\sum_{j=1}^{v^o-1}Z_o^{(j)}=\mu-\mu_o-\Delta\right)}{\pr\!\left(\sum_{j=1}^{v^o}Z_o^{(j)}=\mu-\mu_o\right)}.
\end{aligned}
\end{equation}
Not too surprisingly, we choose $z_i$ and $z_o$ such that
\[
\ex[Z_i]=\frac{\mu-\mu_i}{v^i},\quad \ex[Z_o]=\frac{\mu-\mu_o}{v^o},
\]
or, denoting $\psi_k(z)=zf_{k-1}(z)/f_k(z)$,
\begin{equation}\label{zi,zo,roots}
\ex[Z_i]=\psi_{k_1}(z)=\frac{\mu-\mu_{i}}{v^{i}}\,\,\text{ and }\,\, 
\ex[Z_o]=\psi_{k_2}(z)=\frac{\mu-\mu_{o}}{v^{o}}.
\end{equation}
Intuitively, $\psi_k(z)$ is strictly increasing with $z$. Indeed, for the Poisson $Z=Z(z)$, truncated at $k$, 
\begin{equation}\label{E[Z]grows}
\frac{d \psi_k(z)}{dz}=\frac{d E[Z]}{dz}=z^{-1}\text{Var}(Z)>0.
\end{equation}
(The interested reader may wish to prove this surprisingly simple, yet very useful, identity;  cf.  
 \cite{PitSpeWor}, \cite{AroFriPit}.) Therefore, as $\psi_k(0+)=k$ and $\psi_k(\infty)=\infty$, the equations \eqref{zi,zo,roots} have unique positive roots iff $\mu-\mu_{i}>k_1 v^{i}$ and $\mu-\mu_o>k_2 v^o$, cf. \eqref{edgedensity>}. We will assume that $\bold s$ is such
that, a bit stronger,
\begin{equation}\label{diff>}
\mu-\mu_i-k_1v^i\ge \omega,\quad \mu-\mu_o-k_2v^o\ge \omega,
\end{equation}
where $\omega=\omega(n)\to\infty$ however slowly. From \eqref{zi,zo,roots} it follows
that  $z_i\ge \Theta(\omega/v^i)$, $z_o\ge\Theta(\omega/v^o)$. We also assume that  
\begin{equation}\label{vi,o>,mu<}
v\ge \frac{n}{\log n}.
\end{equation}
Since $\mu\le m=[c_nn]$, we see then that  the parameters $z_{i}$, $z_{o}$ are of 
order $O(\log n)$.
Under the conditions \eqref{diff>}, \eqref{vi,o>,mu<},  the denominators in \eqref{X,Y:Zi,Zo} are given
by a local limit theorem (LLT)
\[
\pr\left(\sum_{j=1}^{v^{i,o}}Z_{i,o}^{(j)}=\mu-\mu_{i,o}\right)
=\frac{1+O((v^{\text{i,o}}z_{\text{i,o}})^{-1})}{\sqrt{2\pi v^{\text{i,o}}\text{Var}(Z_{\text{i,o}})}}
=\theta\bigl((v^{i,o}z_{i,o})^{-1/2}\bigr).
\]
We omit the proof since it is a direct extension of the LLT for the Poissons truncated at 
$k=2$ established in  \cite{AroFriPit}. It follows then that, uniformly over $\delta$,
$\Delta$,
\begin{equation}\label{PX,PY<}
\begin{aligned}
\pr(X=\delta)&=O\bigl(\!\pr(Z_i=\delta)(v^{i}z_{i})^{1/2}\bigr),\\
\pr(Y=\Delta)&=O\bigl(\!\pr(Z_o=\Delta)(v^{o}z_{o})^{1/2}\bigr).
\end{aligned}
\end{equation}
We will also need a sharp asymptotic formula for the ratios of the local probabilities in 
\eqref{X,Y:Zi,Zo}, considerably stronger than a formula obtained by using the LLT 
for the numerators, and separately for the denominators. A very similar case of the Poissons truncated at $k=2$ was analyzed in \cite{PitWor}, Lemma 7. 
\begin{Lemma}\label{ratioappr} Suppose that $\delta=o(v^i)$ and $\Delta=o(v^o)$. Under the
conditions \eqref{diff>}, \eqref{vi,o>,mu<}, the ratios in \eqref{X,Y:Zi,Zo} are $1+O(\delta/v^i)$ and $1+O(\Delta/v^o)$,
so that
\begin{equation}\label{PX,PYclose}
\begin{aligned}
\pr(X=\delta)&=\!\pr(Z_i=\delta)\bigl[1+O(\delta/v^i)\bigr],\\
\pr(Y=\Delta)&=\!\pr(Z_o=\Delta)\bigl[1+O(\Delta/v^o)\bigr].
\end{aligned}
\end{equation}
\end{Lemma}
\noindent
The proof is omitted, as it runs very close to that in \cite{PitWor}, (pp. $154$-$156$). Since
$z_{i,o}=O(\log n)$ and $v_{i,o}\ge n/\log n$, we obtain

\begin{Corollary}\label{E[X],E[Y]=appr} Under the condition \eqref{diff>},
\[
E[X]=\ex[Z_i]+O\bigl(n^{-1}\log^3 n\bigr),\quad E[Y]=\ex[Z_o]+O\bigl(n^{-1}\log^3 n\bigr).
\]
\end{Corollary}

Finally,  since $\ex[Z_{i,o}]=O(\log n)$ as well, we should expect chances of $X$, $Y$ exceeding $\log^2 n$ be very
small. Indeed, applying \eqref{PX,PY<}, for $U=X,Y$,
\begin{equation}\label{P(X>)<,P(Y>)<}
\begin{aligned}
\pr(U\ge& \log^2 n\boldsymbol| \bold V,\mu, u)=O\bigl(e^{-\log^2 n}\bigr)\\&\Longrightarrow
\pr(U\ge \log^2 n\boldsymbol| \bold s)=O\bigl(e^{-\log^2 n}\bigr),
\end{aligned}
\end{equation}
Thus we have a complete description of the distribution of the in/out degree sequence $(\bold X,
\bold Y)$ of $\bold x$, conditioned on $(\bold V,\mu)$. It depends on $\mu$ and $\bold V$,
with the latter entering only through $\bold v$, composed of cardinalities of set-components of $\bold V$, which makes this distribution equal to the distribution of $(\bold X,\bold Y)$ conditioned only on $\bold s$. We had observed already that, conditioned on the in/out degree sequence,
the directed edges are obtained by allocating uniformly at random all $\mu=\sum_i X_i$ in-balls and all $\mu=\sum_iY_i$
out balls among $\mu$ boxes, one in-ball and one out-ball per box. This allows to write a formula for the joint distribution of the numbers of edges between a given vertex $u$ and each  
vertex $j\in [n]$. Here it is. Let $E(u,j)$ and $E(j,u)$ denote the
random number of directed edges from $u$ to $j$,  and from $j$ to $u$; let $\bold E(u)=\{E(u,j),
E(j,u)\}_{j\in [n]}$. Then, for each $\bold e(u):=\{e(u,j),e(j,u)\}_{j\in [n]}$, such that $\sum_j e(u,j)=Y_u$, $\sum_j e(j,u)=X_u$,
we have
\begin{equation}\label{edgedistr}
\begin{aligned}
\pr(\bold E(u)=\bold e(u)| \bold X,\bold Y)&=\frac{\prod_j\binom{X_j}{e(u,j)}\binom{Y_j}{e(j,u)}}
{\sum_{\bold e'}\prod_{w}\binom{X_w}{e'(u,w)}\binom{Y_w}{e'(w,u)}}\\
&=\frac{\prod_{j}\binom{X_j}{e(u,j)}\binom{Y_j}{e(j,u)}}{\binom{\mu}{X_u}\binom{\mu}{Y_u}};
\end{aligned}
\end{equation}
(needless to say, the generic $\bold e'(u)$ in the sum meets the same constraint as $\bold e(u)$).
Consequently
\begin{equation}\label{margE(u,v)}
\begin{aligned}
\pr(E(u,j)=e(u,j)\boldsymbol |\bold X,\bold Y)&=\frac{\binom{X_j}{e(u,j)}\binom{\mu-X_j}{Y_u-e(u,j)}}
{\binom{\mu}{Y_u}},\\
\pr(E(j,u)=e(j,u)\boldsymbol |\bold X,\bold Y)&=\frac{\binom{Y_j}{e(j,u)}\binom{\mu-Y_j}{X_u-e(j,u)}}
{\binom{\mu}{X_u}},
\end{aligned}
\end{equation}
$E(u,j)$, $E(j,u)$ being (conditionally) independent.

{\bf Step 2.\/} We are ready now to evaluate the expected, one-step, change $d\bold s:= \bold s^\prime -\bold s$, conditioned on $\bold s$. As we recall, during a generic step we select a light vertex $u$ uniformly at random and delete all the edges incident to it. Every choice of $u$ and
deletion of the associated edges,  has the same (conditional) probability $1/L(\bold s)$, where $L(\bold s)$, the total number of {\it non-isolated\/} light
vertices at state $\bold s$,  is given by
\[
L(\bold s)=\sum_{(a,b) \neq (0,0)}v_{a,b} +\sum_a v_{a,\bullet}+\sum_b v_{\bullet,b}.
\]
Recall that for $a<k_1$, $b<k_2$, $v_{a,\bullet}$ ($v_{\bullet ,b}$ resp.) is the number of
vertices of in-degree $a$, and out-degree $\ge k_2$ (of out-degree $b$, and in-degree $\ge k_1$
resp.). Clearly it suffices then to evaluate $\ex[d\bold s|\bold s,u]$, i.e.
the expected change conditioned on $\bold s$, and $u$, the chosen light vertex. In each of the
steps that follow, we first derive the expected change of a component of $\bold s'-\bold s$ conditioned on the finer information given by $(\bold X,\bold Y)$ and then average the result
using the distribution of $(\bold X,\bold Y)$, conditional on $\bold s$, obtained in the previous
step.\\

{\bf (1)\/} Let us evaluate $\ex[dv_{a,b}|\bold s,u]$. Clearly, 
\begin{equation*}
\begin{aligned}
v^\prime_{a,b}=&\bigl|\{j:  X_j^\prime=a, Y_j^\prime=b\}\bigr| = \sum_j 1_{\{X^\prime_j=a, Y^\prime_j=b\}}\\
=& v_{a,b} + \sum_j \left( 1_{\{X^\prime_j=a, Y^\prime_j=b\}} -1_{\{X_j=a, Y_j=b\}}\right), 
\end{aligned}
\end{equation*}
so that 
\begin{equation}\label{dvab,uneqv,1}
dv_{a,b}:=v^\prime_{a,b}-v_{a,b}=\sum_j \left( 1_{\{X^\prime_j=a, Y^\prime_j=b\}} -1_{\{X_j=a, Y_j=b\}}\right).
\end{equation}
Here
\begin{equation*}
1_{\{X^\prime_j=a, Y^\prime_j=b\}} -1_{\{X_j=a, Y_j=b\}} = \left\{
\begin{array}{ll}
      1 & j \in V^\prime_{a,b} \text{ and } j \notin V_{a,b} \\
      -1 & j \in V_{a,b} \text{ and } j \notin V^\prime_{a,b}  \\
      0 & \text{ otherwise. }\\
\end{array} 
\right.
\end{equation*}
Therefore 
\begin{equation*}
\begin{aligned}
\ex\bigl[1_{\{X^\prime_j=a, Y^\prime_j=b\}} -1_{\{X_j=a, Y_j=b\}} \boldsymbol | \bold V,\mu,
u\bigr] & = \pr ( j \in V^\prime_{a,b};\, j \notin V_{a,b}\boldsymbol | \bold X,\bold Y,
u ) \\
& - \pr (j \in V_{a,b};\,  j \notin V^\prime_{a,b}\boldsymbol | \bold X,\bold Y,
u)
\end{aligned}
\end{equation*}
Let $j \neq u$.  By \eqref{margE(u,v)},
\begin{equation*}
\begin{aligned}
\pr (j \in V_{a,b};\, j \notin V^\prime_{a,b}  \boldsymbol | & \bold V,\mu,
u)  = 1_{\{j\in V_{a,b}\}}\bigl[1-\pr(j\in V'_{a,b}\boldsymbol |\bold V,\mu,u)\bigr]\\
&=1_{\{j\in V_{a,b}\}}\bigl[1-\pr(E(u,j)=E(j,u)=0\boldsymbol |\bold V,\mu,u)\bigr]\\
&=1_{\{j\in V_{a,b}\}}\left[1-\frac{\binom{\mu-X_u}{b}\binom{\mu-Y_u}{a}}{\binom{\mu}{b}
\binom{\mu}{a}}\right].
\end{aligned}
\end{equation*}
Furthermore
\begin{multline*}
\pr ( j \in V^\prime_{a,b};\,  j \notin V_{a,b}\boldsymbol | \bold V,\mu, u )\\
=\!\!\!\!\! \sum_{k, \ell \geq 0 \atop (k, \ell) \neq (0,0)}\!\!\!\!\!\!\!\pr(X_j = a+k; Y_j = b+ \ell;\, E(j, u)=\ell;\, E(u,j)=k \boldsymbol | \bold V,\mu, u) \\
= \sum_{k, \ell \geq 0 \atop (k, \ell) \neq (0,0)}
 \frac{ {X_u \choose \ell} {\mu-X_u \choose Y_j - \ell}}{{\mu \choose Y_j}} \cdot\frac{ {Y_u \choose k} {\mu-Y_u \choose X_j - k}}{{\mu \choose X_j}} 1_{\{X_j = a+k, Y_j = b+\ell\}}  \\
 = \sum_{k, \ell \geq 0 \atop (k, \ell) \neq (0,0)}
  \frac{ {X_u \choose \ell} {\mu-X_u \choose b}}{{\mu \choose b+\ell}} \cdot\frac{ {Y_u \choose k} {\mu-Y_u \choose a}}{{\mu \choose a+k}} 1_{\{X_j = a+k, Y_j = b+\ell\}}.
 \end{multline*}
So, for $j\neq u$,
\begin{multline}\label{a,b,jnequ}
\ex\bigl[1_{\{X^\prime_j=a, Y^\prime_j=b\}} -1_{\{X_j=a, Y_j=b\}} \boldsymbol | \bold V,\mu,
u\bigr] =-1_{\{j\in V_{a,b}\}}\left[1-\frac{\binom{\mu-X_u}{b}\binom{\mu-Y_u}{a}}{\binom{\mu}{b}
\binom{\mu}{a}}\right]\\
+\!\!\!\sum_{k, \ell \geq 0 \atop (k, \ell) \neq (0,0)}
  \frac{ {X_u \choose \ell} {\mu-X_u \choose b}}{{\mu \choose b+\ell}} \cdot\frac{ {Y_u \choose k} {\mu-Y_u \choose a}}{{\mu \choose a+k}} 1_{\{X_j = a+k, Y_j = b+\ell\}}.
\end{multline}
For $j=u$, we have $1_{\{X^\prime_u=a, Y^\prime_u=b\}} = 1_{\{(a,b) = (0,0)\}}$, since we delete all of $u$'s incident edges.  
Therefore, 
\begin{equation*}
1_{\{X^\prime_u=a, Y^\prime_u=b\}} -1_{\{X_u=a, Y_u=b\}} = 1_{\{(a,b)=(0,0)\}} - 
1_{\{u \in V_{a,b}\}},
\end{equation*}
i.e. {\it conditionally\/} a constant; the last indicator is zero for $(a,b)=(0,0)$, since $u$ is
non-isolated.  So putting together \eqref{dvab,uneqv,1}, \eqref{a,b,jnequ},
the last formula and using notation $V_{\alpha,\beta}:=\{j:\,X_j=\alpha,Y_j=\beta\}$, $v_{\alpha,\beta} := |V_{\alpha, \beta}|$ even for
$\alpha\ge k_1$ or/and $\beta\ge k_2$, we have
\begin{multline}\label{dva,b;X,Y,u}
\ex\bigl[dv_{a,b}\boldsymbol |\bold V,\mu,u\bigr]=1_{\{(a,b)=(0,0)\}} - 1_{\{u \in V_{a,b}\}}\\
-\bigl(v_{a,b}-1_{\{u\in V_{a,b}\}}\bigr)\left[1-\frac{\binom{\mu-X_u}{b}\binom{\mu-Y_u}{a}}{\binom{\mu}{b}\binom{\mu}{a}}\right]\\
+\sum_{k, \ell \geq 0 \atop (k, \ell) \neq (0,0)}
  \frac{ {X_u \choose \ell} {\mu-X_u \choose b}}{{\mu \choose b+\ell}} \cdot\frac{ {Y_u \choose k} {\mu-Y_u \choose a}}{{\mu \choose a+k}} \cdot\bigl(v_{a+k,b+\ell}-1_{\{u\in V_{a+k,b+\ell}\}}\bigr).
 \end{multline}
Now, according to \eqref{P(X>)<,P(Y>)<}, 
\begin{equation}\label{X,Ysmall}
\pr\bigl(\bigl.\max_j(X_j+Y_j)\le \log^2 n\bigr|\,\bold V,\mu\bigr)\ge 1 - e^{-0.5\log^2 n}.
\end{equation}
For the extremely likely values of $(\bold X,\bold Y)$, that we focus on from now, we have
 \[
\frac{bX_u}{\mu}+\frac{aY_u}{\mu}+O\bigl(n^{-2}\log^6n\bigr),
\]
as $\mu$ is $n/\log n$ at least. Further, the dominant contribution to the sum in \eqref{dva,b;X,Y,u} comes from
$(k,\ell)=(1,0)$ and $(k,\ell)=(0,1)$, and the full sum equals
\begin{multline*}
\frac{(b+1)X_u}{\mu}\bigl(v_{a,b+1}-1_{\{u\in V_{a,b+1}\}}\bigr)\\
+\frac{(a+1)Y_u}{\mu}\bigl(v_{a+1,b}-1_{\{u\in V_{a+1,b}\}}\bigr)+O(\eps_n),\quad 
\eps_n:=n^{-1}\log^6 n,
\end{multline*}
as $v_{a+k,b+\ell}\le \mu/(a+b+1)$. Now $\eps_n$ far exceeds the expected contribution of
$(\bold X,\bold Y)$ that do not meet the constraint \eqref{X,Ysmall}. So the equation 
\eqref{dva,b;X,Y,u} becomes
\begin{multline}\label{dva,b;X,Y,u, asym}
\ex\bigl[dv_{a,b}\boldsymbol |\bold V,\mu,u\bigr]=1_{\{(a,b)=(0,0)\}} - 1_{\{u \in V_{a,b}\}}\\
-v_{a,b}\left(\frac{bX_u}{\mu}+\frac{aY_u}{\mu}\right)
+\frac{(b+1)X_u}{\mu}\,v_{a,b+1}
+\frac{(a+1)Y_u}{\mu}\,v_{a+1,b}\\
+O(\eps_n);
\end{multline}
(three terms containing $1_{\{u\in V_{a,b}\}}$, $1_{\{u\in V_{a,b+1}\}}$ and $1_{\{u\in V_{a+1,b}\}}$ added up to $(a+b)/\mu$, absorbed  by $O(\eps_n)$).

Next we use \eqref{dva,b;X,Y,u, asym} to evaluate $\ex\bigl[dv_{a,b}\boldsymbol |\bold V,\mu\bigr]$, recalling that $u$ is chosen uar from  $L:=L(\bold s)$ non-isolated light vertices. To do so,
notice first that
\[
\sum_{j\text{ light}}X_j=\mu_i+\sum_{j\in J}X_j1_{\{Y_j<k_2\}},
\]
where $\mu_i$ is the total in-degree of in-light vertices, determined completely by $\bold s$,
see \eqref{sumsdelDel}. So
\begin{align}
\ex\left[\sum_{j\text{ light}}X_j\biggr|\bold V,\mu\right]&=\mu_i+
\ex\left[\sum_{j\in J}X_j1_{\{Y_j<k_2\}}\biggr|\bold V,\mu\right]\notag\\
&=\mu_i+\ex[X]\big|\{j\in J: Y_j<k_2\}\bigr|\notag\\
&=\mu_i+\ex[X] (v^i-v)=:E_i;\label{=:Ei}
\end{align}
explanation: $J=J(\bold V,\mu)$ is the set of all $v^i$ in-heavy vertices; conditioned on
$(\bold V,\mu)$,  $X$ has the common distribution of the in-degree of each one of those vertices, see \eqref{X,Y}; $v$ is the number of all in/out-heavy vertices. Analogously
\begin{equation}\label{=:Eo}
\ex\left[\sum_{j\text{ light}}Y_j\biggr|\bold V,\mu\right]=\mu_o+\ex[Y](v^o-v)=:E_o.
\end{equation}
Now for  $a\le k_1-2$, $b\le k_2-2$ the parameters $v_{a+1,b}$, $v_{a,b+1}$ are
completely determined by $\bold s$. So, using \eqref{=:Ei} and \eqref{=:Eo}, we have
\begin{multline}\label{dva,b;X,Y, asym}
\ex\bigl[dv_{a,b}\boldsymbol |\bold V,\mu\bigr]=1_{\{(a,b)=(0,0)\}}-\frac{v_{a,b}}{L}\,1_{\{(a,b)\neq
(0,0)\}}\\
+\frac{E_o}{L}\,\frac{(a+1)v_{a+1,b}-av_{a,b}}{\mu}+\frac{E_i}{L}\,\frac{(b+1)v_{a,b+1}-b v_{a,b}}{\mu}\\
+O\bigl(n^{-1}\log^2 n\bigr).
\end{multline}
We still need to consider the border values $a=k_1-1$ and/or $b=k_2-1$, in which case, given
$\bold s$, $v_{a+1,b}$ and $v_{a,b+1}$ are random. 

Introducing $V^{\alpha}=\cup_{\beta}V_{\alpha,\beta}\cup V_{\alpha,\bullet}$, 
$V^{\beta}=\cup_{\alpha}V_{\alpha,\beta}\cup V_{\bullet,\beta}$, where $\alpha<k_1$, $\beta<k_2$, 
we evaluate
\begin{align*}
v_{a,k_2}\cdot \sum_{j\text{ light}}X_j&=\left(\sum_{\ell\in V^a\cap I} 1_{\{Y_{\ell}=k_2\}}\right)
\left(\mu_i+\sum_{j\in J\cap\,(\cup_{\beta} V^{\beta})}\!\!X_j\right)\\
&=\mu_i\sum_{\ell\in V^a\cap I} 1_{\{Y_{\ell}=k_2\}}+\sum_{\ell\in V^a\cap I,\atop
j\in J\cap\,(\cup_{\beta} V^{\beta})}\!\!\!
X_j1_{\{Y_{\ell}=k_2\}}.
\end{align*}
In the second sum, given $\bold V$ and $\mu$, $X_j$ and $1_{\{Y_{\ell}=k_2\}}$ are independent.
Besides,  $\bigl|J\cap\,(\cup_{\beta} V^{\beta})\bigr|=v^i-v$, i.e.
the total number of in-heavy/out-light vertices, and $\bigl|V^a\cap I\bigr|=v_{a,\bullet}$, i.e.
the total number of out-heavy vertices, with light in-degree $a$.
Therefore
\begin{equation}\label{Evak2,u}
\begin{aligned}
\ex\left[v_{a,k_2}\cdot \sum_{j\text{ light}}X_j\biggr|\bold V,\mu\right]&
=\bigl(\mu_i+\ex[X](v^i-v)\bigr)\!\!\pr(Y=k_2)v_{a,\bullet}\\
&=E_i\!\!\pr(Y=k_2)v_{a,\bullet}.
\end{aligned}
\end{equation}
Analogously
\begin{equation}\label{Evk1b}
\ex\left[v_{k_1,b}\cdot \sum_{j\text{ light}}Y_j\biggr|\bold V,\mu\right]=E_o\!\!\pr(X=k_1)v_{\bullet,b}.
\end{equation}
These two identities mean that the equation \eqref{dva,b;X,Y, asym} holds for $a=k_1-1$ and
$b=k_2-1$, if we {\it define\/}
\[
v_{k_1,b}=\pr(X=k_1)v_{\bullet,b},\quad v_{a,k_2}=\pr(Y=k_2)v_{a,\bullet}.
\]
The remainder term aside, the expression in \eqref{dva,b;X,Y, asym} depends on $(\bold V,
\mu)$ only through $\bold s$. So we can, and will  replace conditioning on $(\bold V,\mu)$ by conditioning on $\bold s$ only.\\

{\bf (2)\/} Now let us turn to $v^\prime_{a, \bullet}$ and $v^\prime_{\bullet, b}$. 
Just as we've done with $v^\prime_{a,b}$, we write
\begin{equation}
dv_a := v^\prime_{a,\bullet} - v_{a, \bullet} =  \sum_j \left(1_{\{X^\prime_j = a, Y^\prime_j \geq k_2\}} - 1_{\{ X_j = a, Y_j \geq k_2 \}} \right),
\end{equation}
where
\begin{equation*}
1_{\{X^\prime_j=a, Y^\prime_j\geq k_2\}} -1_{\{X_j=a, Y_j \geq k_2\}} = \left\{
\begin{array}{ll}
      1 & j \in V^\prime_{a,\bullet}\text{ and } j \notin V_{a,\bullet} ,\\
      -1 & j \in V_{a,\bullet},\text{ and }j \notin V^\prime_{a,\bullet} , \\
      0 & \text{ otherwise. }\\
\end{array} 
\right.
\end{equation*}
So
\begin{multline*}
\ex\bigl[1_{\{X^\prime_j=a, Y^\prime_j\geq k_2\}} -1_{\{X_j=a, Y_j\geq k_2\}} \boldsymbol | \bold V,\mu,u\bigr]\\
 =1_{\{ j \notin V_{a,\bullet}\}}\!\pr ( j \in V^\prime_{a,\bullet}\boldsymbol | \bold V,\mu,
u ) 
- 1_{\{j \in V_{a,\bullet}\}}\!\pr ( j \notin V^\prime_{a,\bullet}\boldsymbol | \bold V,\mu, u).
\end{multline*}
Let  $j \neq u$. For $j\in V_{a,\bullet}$, the event $\{ j \notin V^\prime_{a,\bullet}\}$ is a non-disjoint
union of two events, ``there is an edge from $u$ to $j$'' , i.e. its deletion alone will pull the in-degree of $j$ below $a$, and  ``there are sufficiently many edges from $j$ to $u$'', i.e. such that their deletion will pull out-degree of $j$ below $k_2$. Therefore, by the conditional
independence of $E(u,j)$ and $E(j,u)$,
\begin{multline}\label{prob_breakdown_va}
\pr ( j \notin V^\prime_{a,\bullet}  \boldsymbol |  \bold V,\mu,
u) = \!\pr \bigl( E(u,j) \geq 1\boldsymbol | \boldsymbol V,\mu, u) 
 +\! \!\pr\bigl(E(u,j)=0\boldsymbol|\bold V,\mu,u\bigr)\\
 \times\left[\sum_{k\ge 0,\ell>0}\!\!1_{\{Y_j=k_2+k\}}\!\!\pr \bigl( E(j,u) = k+\ell \boldsymbol | \bold V,\bold \mu, u\bigr)\!\!\right]\!;
\end{multline}
here, by \eqref{margE(u,v)}, for the moderate $(\bold X,\bold Y)$,
\[
\pr \bigl( E(u,j) \geq 1\boldsymbol | \boldsymbol V,\mu, u)=1-\frac{\binom{\mu-Y_u}{a}}{\binom{\mu}
{a}}=\frac{aY_u}{\mu}+O\bigl(n^{-2}\log^6n)\bigr).
\]
Furthermore, the sum over $k,\ell$ equals
\begin{equation*}
 \sum_{k \geq 0, \ell \geq 1}\!\! 1_{\{Y_j=k_2+k}\}\frac{{X_u \choose k+\ell}{\mu-X_u \choose (k_2+k)-(k+\ell)}}{{\mu \choose k_2+k}} 
 =\frac{k_2 X_u}{\mu}  1_{\{ Y_j = k_2\}} + O\bigl( n^{-6}\log^2 n \bigr) 1_{\{Y_j \geq k_2\}}.
\end{equation*}
Hence, for $j\in V_{a,\bullet}$,
\begin{equation*}
\pr (j \notin V^\prime_{a,\bullet}  \boldsymbol | \bold V,\mu, u) = \frac{a Y_u}{\mu} + \frac{k_2 X_u}{\mu}  1_{\{Y_j = k_2\}} + O\bigl(n^{-2} \log^6n\bigr) 1_{\{Y_j \geq k_2\}}.
\end{equation*}
Analogously
\begin{equation*}
\begin{aligned}
\pr ( j \in V^\prime_{a,\bullet};\,& j \notin V_{a,\bullet}\boldsymbol | \bold X,\bold Y, u ) = \sum_{k \geq 1} 1_{\{X_j = a+k, \,Y_j\ge k_2\}}\!\pr\bigl(E(u,j)=k \boldsymbol | \bold V,\mu, u\bigr) \\&= \sum _{k\ge 1}\frac{ {Y_u \choose k} {\mu-Y_u \choose a}}{{\mu \choose a+k}} 1_{\{X_j = a+k,\, Y_j \geq k_2\}}  \\
& = \frac{(a+1)Y_u}{\mu} 1_{\{X_j = a+1, \,Y_j \geq k_2\}} + O\bigl(n^{-2} \log^6n\bigr) 1_{\{ X_j \geq a+1, \,Y_j \geq k_2\}}. 
\end{aligned}
\end{equation*}
Therefore for $j\neq u$, 
\begin{multline*}
\ex\bigl[1_{\{X^\prime_j=a, \,Y^\prime_j\geq k_2\}} -1_{\{X_j=a, \,Y_j\geq k_2\}} \boldsymbol | \bold V,
\mu, u\bigr] \\
 =\frac{(a+1)Y_u}{\mu} 1_{\{X_j = a+1,\, Y_j \geq k_2\}} - \frac{a Y_u}{\mu} 1_{\{X_j = a, \,Y_j \geq k_2\}}
 - \frac{k_2 X_u}{\mu}  1_{\{X_j = a, \,Y_j = k_2\}}\\
 + O\bigl(n^{-2} \log^2n\bigr)  1_{\{X_j \geq  a, \,Y_j \geq k_2\}}.
\end{multline*}
Adding up these equations for $j\neq u$,  we obtain 
\begin{multline*}
\ex\left[ \sum_{j \neq u} 1_{\{X^\prime_j=a, \,Y^\prime_j\geq k_2\}} -1_{\{X_j=a,\, Y_j\geq k_2\}} \boldsymbol | \bold V,\mu, u\right] \\
 =\frac{(a+1)Y_u}{\mu}\, v_{a+1, \bullet} - \frac{a Y_u}{\mu}\, v_{a, \bullet}
- \frac{k_2 X_u}{\mu}\,  v_{a, k_2} 
+ O\bigl( \eps_n);
\end{multline*}
like earlier, $v_{k_1,\bullet}$ is defined as the number of vertices $j$ with in-degree $X_j=k_1$ and out-degree $Y_j\ge k_2$.

For $j=u$, since we delete all of $u'$s incident edges, 
$
1_{\{X^\prime_u = a, \,Y^\prime_u \geq k_2\}} = 0, 
$
so that
\begin{equation*}
1_{\{X^\prime_u = a,\, Y^\prime_u \geq k_2\}}-1_{\{X_u = a, \,Y_u \geq k_2\}} = -1_{\{X_u = a, \,Y_u \geq k_2\}}=-1_{\{u\in V_{a,\bullet}\}}.
\end{equation*}
Hence, 
\begin{multline*}
\ex\bigl[dv_{a,\bullet}\boldsymbol|\bold V,\mu,u]=\ex\left[ \sum_{j} 1_{\{X^\prime_j=a, \,Y^\prime_j\geq k_2\}} -1_{\{X_j=a,\, Y_j\geq k_2\}}  \bigg | \bold V,\mu, u\right] \\
=-1_{\{ u \in V_{a, \bullet}\}}
+ \frac{(a+1)Y_u}{\mu}\,v_{a+1, \bullet} - \frac{a Y_u}{\mu}\, v_{a, \bullet}
 - \frac{k_2 X_u}{\mu} \, v_{a, k_2} + O(\eps_n).
\end{multline*}
It remains to average this identity over $u$. 
For $a\le k_1-2$, 
 \begin{equation}\label{a<k1-2}
\ex\left[\left(\sum_{j\text{ light}}Y_j\right)v_{a+1,\bullet}\biggr|\bold V,\mu\right]=
E_o v_{a+1,\bullet},
\end{equation}
with $E_o$ defined in \eqref{=:Eo}.  Now let $a=k_1-1$. Then, introducing $V_{\bullet,\bullet}$,
the set of all doubly heavy vertices,
\begin{multline*}
\left(\sum_{j\text{ light}}Y_j\right)v_{a+1,\bullet}=\left(\mu_o+\sum_{\alpha<k_1}\sum_{j\in V_{\alpha,\bullet}}Y_j\right)\sum_{\ell\in V_{\bullet,\bullet}}1_{\{X_{\ell}=k_1\}}\\
=\mu_o\sum_{\ell\in V_{\bullet,\bullet}}1_{\{X_{\ell}=k_1\}}+\sum_{\alpha<k_1}\,\sum_{j\in V_{\alpha,
\bullet},\,\ell\in V_{\bullet,\bullet}}Y_j1_{\{X_{\ell}=k_1\}}.
\end{multline*}
Therefore, using $|V_{\alpha,\bullet}|=v_{\alpha,\bullet}$ for $\alpha<k_1$, $|V_{\bullet,\bullet}|=v$, we have: for $a=k_1-1$
\begin{align*}
\ex\left[\left(\sum_{j\text{ light}}Y_j\right)v_{a+1,\bullet}\biggr|V,\mu\right]&=\mu_ov\!\pr(X=k_1)
+v\!\pr(X=k_1)\ex[Y] \sum_{\alpha<k_1}v_{\alpha,\bullet}\\
&=\bigl(\mu_o+\ex[Y](v^o-v)\bigr)v\!\pr(X=k_1)=E_ov\!\pr(X=k_1).
\end{align*}
Likewise for $a\le k_1-1$
\begin{multline*}
\ex\left[\left(\sum_{j\text{ light}}X_j\right)v_{a,k_2}\biggr|V,\mu\right]=
\bigl(\mu_i+\ex[X](v^i-v)\bigr)v_{a,\bullet}\pr(Y=k_2) \\ =E_iv_{a,\bullet}\pr(Y=k_2),
\end{multline*}
and of course
\[
\ex\left[\left(\sum_{j\text{ light}}Y_j\right)v_{a,\bullet}\biggr|V,\mu\right]=E_ov_{a,\bullet},
\]
as $v_{a,\bullet}$ is constant, given $\bold V$, $\mu$. Absorbing the negligible expected
contribution of $(\bold X,\bold Y)$ violating the condition \eqref{X,Ysmall}, we conclude that, for all $a\le k_1-1$,
\begin{multline}\label{Edva,bull,V,mu}
\ex\bigl[dv_{a,\bullet}\boldsymbol|\bold V,\mu]=-\frac{v_{a,\bullet}}{L}\\
+\frac{E_o}{L}\,\frac{(a+1)v_{a+1,\bullet}-av_{a,\bullet}}{\mu}-\frac{E_i}{L}\frac{k_2v_{a,k_2}}{\mu}
+O(\eps_n),
\end{multline}
where $v_{k_1,\bullet}:=v\pr(X=k_1)$, $v_{a,k_2}:=v_{a,\bullet}\pr(Y=k_2)$. And we have
a similar expression for $\ex\bigl[dv_{\bullet,b}\boldsymbol|\bold V,\mu]$
with $v_{\bullet,k_2}:=v\!\pr(Y=k_2)$, $v_{k_1,b}:=v_{\bullet, b}\!\pr(X=k_1)$.\\

\begin{Lemma}\label{Edvab,va,bull,vbull,b,s} With $\eps_n=n^{-1}\log^6n$,
\begin{align}
\ex\bigl[dv_{a,b}\boldsymbol |\bold s \bigr]&=1_{\{(a,b)=(0,0)\}}-\frac{v_{a,b}}{L}\,1_{\{(a,b)\neq (0,0)\}}
+\frac{E_o}{L}\,\frac{(a+1)v_{a+1,b}-av_{a,b}}{\mu}\notag\\
&+\frac{E_i}{L}\,\frac{(b+1)v_{a,b+1}-b v_{a,b}}{\mu}+O(\eps_n),\label{edvab}\\
\ex\bigl[dv_{a,\bullet}\boldsymbol|\bold s]&=-\frac{v_{a,\bullet}}{L}\notag\\
&+\frac{E_o}{L}\,\frac{(a+1)v_{a+1,\bullet}-av_{a,\bullet}}{\mu}-\frac{E_i}{L}\frac{k_2v_{a,k_2}}{\mu}
+O(\eps_n),\label{edva}\\
\ex\bigl[dv_{\bullet,b}\boldsymbol|\bold s]&=-\frac{v_{\bullet,b}}{L}\notag\\
&+\frac{E_i}{L}\,\frac{(b+1)v_{\bullet,b+1}-bv_{\bullet,b}}{\mu}-\frac{E_o}{L}\frac{k_1v_{k_1,b}}{\mu}
+O(\eps_n),\label{edvb}\\
\ex[dv\boldsymbol |\bold s]&=-\frac{E_o}{L}\,\frac{k_1v_{k_1,\bullet}}{\mu}-\frac{E_i}{L}\,
\frac{k_2v_{\bullet,k_2}}{\mu}+O(\eps_n),\label{Edv}\\
\ex[d\mu\boldsymbol |\bold s]&=-\frac{E_i}{L}-\frac{E_o}{L}+O(\eps_n).
\label{Edmu}
\end{align}
Here $L=L(\bold s)$ is the total number of non-isolated light vertices, i.e.
\begin{equation}\label{L(t)=}
L = \sum_{a < k_1, b<k_2, \atop (a,b) \neq (0,0)} v_{a,b} + \sum_{0\leq a < k_1} v_{a, \bullet} + \sum_{0 \leq b < k_2} v_{\bullet, b},
 \end{equation}
and
\begin{equation}\label{extrav}
\begin{aligned}
&v_{k_1,b}=v_{\bullet,b}\!\pr(X=k_1)\,, v_{a,k_2}:=v_{a,\bullet}\!\pr(Y=k_2),\\
&v_{\bullet,k_2}=v\!\pr(Y=k_2),\,, v_{k_1,\bullet}=v\!\pr(X=k_1),\\
&E_i=\mu_i+E[X](v^i-v),\quad E_o=\mu_o+E[Y](v^o-v),
\end{aligned}
\end{equation}
\end{Lemma}

\begin{proof} The first three formulas follow from \eqref{dva,b;X,Y,u, asym}, \eqref{Edva,bull,V,mu},
and the latter's counterpart  for $\ex[dv_{\bullet,b}\boldsymbol | \bold V,\mu]$, by replacing 
conditioning on $(\bold V,\mu)$ with that on $\bold s$. The proof of the last two equations is
similar and is omitted.
\end{proof}

\begin{Remark} As for the {\it likely\/} bounds  for $dv_{a,b},\dots, d\mu$,
it is clear  that, given $\bold V,\mu, u$, their absolute values are each $X_u+Y_u$ at most.
So, conditioned on an admissible $\bold s$,  the increments 
$|dv_{a,b}|,\dots, d\mu$ are of order $O(\log^2 n)$ with probability 
$\ge 1- e^{-0.5(\log^2 n)\log\log n}$.
\end{Remark}

The total number of variables $v_{a,b}$, $v_{a,\bullet}$, $v_{\bullet,b}$, $v$, $\mu$ is
$(k_1+1)(k_2+1)+1$, fast growing with $k_1$, $k_2$. Observe though that, $L^{-1}(t)$ aside, the RHS in the equations \eqref{Edv}, \eqref{Edmu} depend only on the 
$6$-dimensional $\mathbf R:=(v,v^i,v^o,\mu,\mu_i,\mu_o)$, since $\ex[X], \ex[Y], \pr(X=k_1), 
\pr(Y=k_2)$ are functions of $\mathbf R$ only. Remarkably, it follows from the telescopic
structure of the RHS's in \eqref{edvab}, \eqref{edva} and \eqref{edvb} that the same property holds for the conditional expected changes of $v^{i,o}$ and $\mu_{i,o}$. 

To show this, let us first compute $\ex[dv^a\boldsymbol |\bold s]$. Recall that $v^a = v_{a, \bullet} + \sum_b v_{a,b}$ is the number of (light) vertices with in-degree $a(< k_1)$. So, using 
\eqref{edvab} and \eqref{edva}, 
\begin{align*}
\ex[dv^a&\boldsymbol |\bold s]= - \frac{v_{a, \bullet}+ \sum_b v_{a,b}}{L}\\
&+ \frac{E_o}{L} \frac{(a+1) (v_{a+1, \bullet}+ \sum_b v_{a+1, b} ) - a (v_{a, \bullet} + \sum_b v_{a,b}) }{\mu} +O(\eps_n)\\ &=- \frac{v^a}{L} + \frac{E_o}{L} \frac{(a+1)v^{a+1}-a v^a}{\mu}+O(\eps_n).
\end{align*}
Since $\mu_i = \sum_a a v^a$, we have then
\begin{align}
\ex[d\mu_i&\boldsymbol | \bold s] = \sum_a\left[\frac{- a v^a}{L} + \frac{E_o}{L}\frac{a(a+1)v^{a+1}-a^2v^a}
{\mu}\right]+O(\eps_n)\notag\\
&=-\frac{\mu_i}{L}+\frac{E_o}{L}\sum_a \frac{(a+1) a v^{a+1} - a(a-1) v^a - a v^a}{\mu}+
O(\eps_n)\notag\\ 
&= -\frac{\mu_i}{L} + \frac{E_o}{L} \frac{k_1 (k_1-1) v^{k_1}}{\mu} - \frac{E_o}{L} \frac{\mu_i}{\mu}+O(\eps_n)\notag\\
& = -\frac{\mu_i}{L} \left( 1 + \frac{E_o}{\mu} \right) + \frac{E_o}{L} \frac{k_1(k_1-1) v^{k_1}}{\mu} +O(\eps_n)\notag\\ 
&= -\frac{\mu_i}{L} \left( 1 + \frac{E_o}{\mu} \right) + \frac{E_o k_1 (k_1-1)P(X=k_1) v^i}{L \mu}+O(\eps_n),\label{exdmui}
\end{align}
the last equality following from 
\[
v^{k_1}:=v_{k_1,\bullet}+\sum_bv_{k_1,b}
=\bigl(v+\sum_b v_{\bullet,b}\bigr)\!\pr(X=k_1)=v^i\!\pr(X=k_1).
\]
Exchanging ``i" and ``o", $k_1$ and $k_2$, and $Z_i$ and $Z_o$, we obtain 
\begin{equation}\label{exfmo}
\ex[d\mu_o\boldsymbol | \bold s] =-\frac{\mu_o}{L} \left( 1 + \frac{E_i}{\mu} \right) + \frac{E_i k_2 (k_2-1)P(Y=k_2) v^o}{L \mu}+O(\eps_n).
\end{equation}
Next 
\begin{multline*}
\ex\left[d \sum_b v_{\bullet, b}\biggl|\bold s\right] = - \frac{\sum_b v_{\bullet, b}}{L} + \frac{E_i}{L} \frac{k_2 v_{\bullet, k_2}}{\mu} - \frac{E_o}{L} \frac{k_1 \sum_b v_{k_1, b}}{\mu}+O(\eps_n) \\ = - \frac{\sum_b v_{\bullet, b}}{L} + \frac{E_i \, k_2 P(Y=k_2) v}{L \mu} - \frac{E_o k_1 P(X=k_1) \sum_b v_{\bullet, b}}{L \mu}+O(\eps_n).
\end{multline*}
Therefore
\begin{align}
\ex[dv^i\boldsymbol |\bold s] &= - \frac{\sum_b v_{\bullet, b}}{L} - \frac{E_o k_1 P(Z_i=k_1) v^i}{L \mu}+O(\eps_n)\notag \\&= - \frac{v^i - v}{L} - \frac{E_o k_1 P(X=k_1) v^i}{L \mu}
+O(\eps_n),\label{dvi/dt}
\end{align}
and we have a similar equation for $\ex[dv^o\boldsymbol |\bold s]$. For ease of reference, here is the resulting 
claim.
\begin{Lemma}\label{6eqtns}
\begin{equation}\label{system6}
\begin{aligned}
\ex[dv^i\boldsymbol |\bold s]&=- \frac{v^i - v}{L} - \frac{E_o k_1 P(X=k_1) v^i}{L \mu}+O(\eps_n),\\
\ex[dv^o\boldsymbol |\bold s]&=- \frac{v^o - v}{L} - \frac{E_i k_2 P(Y=k_2) v^0}{L \mu}+O(\eps_n),\\
\ex[d\mu_i\boldsymbol |\bold s] &=-\frac{\mu_i}{L} \left( 1 + \frac{E_o}{\mu} \right) + \frac{E_o (k_1)_2 P(X=k_1) v^i}{L \mu}+O(\eps_n),\\
 \ex[d\mu_o\boldsymbol |\bold s] &=-\frac{\mu_o}{L} \left( 1 + \frac{E_i}{\mu} \right) + \frac{E_i (k_2)_2 P(Y=k_2) v^o}{L \mu}+O(\eps_n),\\
\ex[dv\boldsymbol |\bold s] &
 = -\frac{E_o}{L} \frac{k_1 P(X=k_1) v}{\mu} - \frac{E_i}{L} \frac{k_2 P(Y=k_2) v}{\mu}+O(\eps_n), \\
\ex[d\mu\boldsymbol |\bold s]&= - \frac{E_i}{L} - \frac{E_o}{L}+O(\eps_n).
\end{aligned}
\end{equation}
\end{Lemma}
To be sure, $L=L(\bold s)$, the number of non-isolated light vertices at state $\bold s$, is not a function of $(v^{i,o}, \mu_{i,o}, v,\mu)$ only. Fortunately its role is confined to being a sort of scaling parameter, and
to a substantial degree we will be able to view these equations as describing stochastic
dynamics of the leading parameter $\bold R:=(v^{i,o}, \mu_{i,o}, v,\mu)$. 

Finally, we observe
that, by \eqref{PX,PYclose} and Corollary \ref{E[X],E[Y]=appr}, we can, and will replace $X$ and $Y$ with $Z_i$ and $Z_o$ respectively, at the cost of an extra error term $O(n^{-1}\log n)$, absorbed by $O(\eps_n)=O\bigl(n^{-1}\log^2 n\bigr)$.

\section{Deterministic version}
Excluding the near-terminal moments $t$, the random variables $v^{i,o}(t)$, $\mu_{i,o}(t)$, $v(t)$ and $\mu(t)$ are all of order $n$, while the RHS expressions in \eqref{system6} for their
one-step expected changes are bounded. Intuitively this suggests that a deterministic trajectory
defined as a solution of the system differential equations below is a likely, relatively close, 
approximation of the random deletion process for those $t$'s:
\begin{equation}\label{detsystem}
\begin{aligned}
\frac{d v_{0,0}}{dt} &= 1+ \frac{E_o}{L} \frac{v_{1,0} }{\mu}  + \frac{E_i}{L} \frac{v_{0,1}}{\mu}, \\
\frac{d v_{a, b}}{dt} &= -\frac{v_{a, b}}{L} + \frac{E_o}{L}\frac{(a+1) v_{a+1, b} - a v_{a, b}}{\mu}  + \frac{E_i}{L} \frac{(b+1) v_{a, b+1} - b v_{a, b}}{\mu}, \\
\frac{d v_{a, \bullet}}{dt} &= -\frac{v_{a, \bullet}}{L} + \frac{E_o}{L} \frac{(a+1) v_{a+1, \bullet} - a v_{a, \bullet}}{\mu} - \frac{E_i}{L}\frac{k_2 v_{a, k_2}}{\mu}, \\
\frac{d v_{ \bullet,b}}{dt} &= -\frac{v_{\bullet,b}}{L} + \frac{E_i}{L}\frac{(b+1) v_{ \bullet,b+1} - b v_{\bullet,b}}{\mu}- \frac{E_o}{L} \frac{k_1 v_{k_1,b}}{\mu}, \\
\frac{dv}{dt}&=-\frac{E_o}{L}\frac{k_1v_{k_1,\bullet}}{\mu}-\frac{E_i}{L}\frac{k_2v_{\bullet,k_2}}{\mu}\\
&=-\frac{E_o}{L}\frac{k_1P(Z_i=k_1)v}{\mu}-\frac{E_i}{L}\frac{k_2P(Z_o=k_2)v}{\mu},\\
\frac{d\mu}{dt}&=-\frac{E_i}{L}-\frac{E_o}{L},
\end{aligned}
\end{equation}
where $E_{i,o}\,,E_o$ and the border parameters $v_{\alpha,\beta}$ are defined in
\eqref{extrav}, with $X$, $Y$ replaced by $Z_i$ and $Z_o$. 

We took liberty using the old notations, $v_{a,b}$ etc.,  for these non-random variables. In the next
section we will adorn these functions with a bar, $\bar v_{a,b}$ etc., since our task will be to analyze likely magnitude of $v_{a,b}(t)-\bar v_{a,b}(t)$, etc. at integer $t$, $v_{a,b}(t)$ etc. being  components of the random $\bold s(t)$.

The corresponding system for $\bold R=(v^i,v^o,\mu_i,\mu_o,v,\mu)$ is
\begin{equation}\label{system6}
\begin{aligned}
\frac{dv^i}{dt}&=- \frac{v^i - v}{L} - \frac{E_o k_1 P(Z_i=k_1) v^i}{L \mu},\\
\frac{dv^o}{dt}&=- \frac{v^o - v}{L} - \frac{E_i k_2 P(Z_o=k_2) v^0}{L \mu},\\
\frac{d\mu_i}{dt} &=-\frac{\mu_i}{L} \left( 1 + \frac{E_o}{\mu} \right) + \frac{E_o (k_1)_2 P(Z_i=k_1) v^i}{L \mu},\\
 \frac{d\mu_o}{dt} &=-\frac{\mu_o}{L} \left( 1 + \frac{E_i}{\mu} \right) + \frac{E_i (k_2)_2 P(Z_o=k_2) v^o}{L \mu},\\
\frac{dv}{dt}&
 = -\frac{E_o}{L} \frac{k_1 P(Z_i=k_1) v}{\mu} - \frac{E_i}{L} \frac{k_2 P(Z_o=k_2) v}{\mu},
\\
\frac{d\mu}{dt}&= - \frac{E_i}{L} - \frac{E_o}{L}.
\end{aligned}
\end{equation}
We remind the reader that $Z_i$ and $Z_o$ are $\text{Poi}(z_i)$ and $\text{Poi}(z_o)$ conditioned, respectively, on the events ``$\text{Poi}(z_i)\ge k_1$'' and ``$\text{Poi}(z_o)\ge k_2$'',
with $z_i$, $z_o$ chosen such that
\begin{equation}\label{zi,zo;ultim}
\begin{aligned}
\ex[Z_i]&=\frac{z_iP(\text{Poi}(z_i)\ge k_1-1)}{P(\text{Poi}(z_i)\ge k_1)}=\frac{z_i f_{k_1-1}(z_i)}
{f_{k_1}(z_i)}
=\frac{\mu-\mu_i}{v^i},\\
\ex[Z_o]&=\frac{z_oP(\text{Poi}(z_o)\ge k_2-1)}{P(\text{Poi}(z_o)\ge k_2)}=\frac{z_o f_{k_2-1}(z_o)}
{f_{k_2}(z_o)}
=\frac{\mu-\mu_o}{v^o},\\
f_k(z)&:=\sum_{j\ge k}\frac{z^j}{j!}.
\end{aligned}
\end{equation}
i.e. $z_i,z_o$ are ultimately  functions of $\mathbf R$. Contrary to its intimidating appearance, the system
\eqref{system6} has rather remarkable properties that will enable us to to obtain both
explicit and qualitative results on the trajectories behavior. 

\subsection{Conservation laws} As we are about to see, the rates $dz_{i,o}/dt$ provide the keys.  Using the first line in \eqref{zi,zo;ultim}, we have 
\begin{align*}
\frac{dE[Z_i] }{dt} &= \frac{1}{v^i} \frac{d(\mu-\mu_i)}{dt} - \frac{E[Z_i]}{v^i} \frac{d}{dt} v^i \\&= \frac{1}{v^i} \left( -\frac{E_i}{L}- \frac{E_o}{L}  +\frac{\mu_i}{L} \left( 1 + \frac{E_o}{\mu} \right) - \frac{E_o}{L} \frac{k_1(k_1-1) P(Z_i=k_1) v^{i}}{\mu} \right) \\&- \frac{E[Z_i]}{v^i} \left(- \frac{v^i - v}{L} - \frac{E_o k_1 P(Z_i=k_1) v^i}{L \mu}  \right) \\&= \frac{1}{v^i L} \left( - E_i - E_o + \mu_i + \frac{\mu_i E_o}{\mu} - \frac{E_o k_1 (k_1-1) P(Z_i=k_1) v^i}{\mu} \right. \\& \left.+ E[Z_i] (v^i-v) + \frac{E[Z_i] E_o k_1 P(Z_i=k_1) v^i}{\mu} \right).
\end{align*}
Using again $\ex[Z_i]=\tfrac{\mu-\mu_i}{v^i}$, and $E_i= \mu_i+E[Z_i] (v^i-v)$, we transform
the above expression into 
\begin{align*}
\frac{dE[Z_i] }{dt}& = \frac{E_o}{v^i L} \left( - 1 + \frac{\mu_i}{\mu} - \frac{(k_1)_2 P(Z_i=k_1) v^i}{\mu}+\frac{E[Z_i] k_1 P(Z_i=k_1)v^i}{\mu} \right) \\&= \frac{E_o E[Z_i]}{L \mu} \left( - \frac{(\mu-\mu_i)}{v^i E[Z_i]} - \frac{(k_1)_2 P(Z_i=k_1)}{ E[Z_i]}+k_1 P(Z_i=k_1) \right) \\ &= - \frac{E_o E[Z_i]}{L \mu} \left( 1 - k_1 P(Z_i=k_1) + \frac{(k_1)_2 P(Z_i=k_1)}{ E[Z_i]} \right)\\
&=- \frac{E_o E[Z_i]}{L \mu} \left(1-\frac{z_iP(\text{Poi}(z_i)=k_1-1)}{P(\text{Poi}(z_i)\ge k_1)}+
\frac{z_iP(\text{Poi}(z_i)=k_1-2)}{P(\text{Poi}(z_i)\ge k_1-1)}\right)\\
&=- \frac{E_o z_i E[Z_i]}{L \mu}\,\frac{d}{dz_i}\bigl(\log z_i - \log P(\text{Poi}(z_i) \geq k_1)+\log P(\text{Poi}(z_i)\geq k_1-1)\bigr)\\
&=- \frac{E_o z_i}{L \mu}\,E[Z_i] \frac{d\log E[Z_i]}{dz_i}
=-\frac{E_oz_i}{L\mu}\cdot \frac{d\ex[Z_i]}{dz_i}.
\end{align*}
Of course, the analogous identity holds for $d E[Z_o]/dt$. 

In view of \eqref{E[Z]grows}, implicit in the above sequence of equalities is a general formula 
for $Z=\text{Poi}(z)$ truncated at $k$:
\begin{equation}\label{Var(Z)}
\text{Var}(Z)=\ex[Z](1-\phi_k(z)+\phi_{k-1}(z)),\,\, \phi_r(z):=\frac{z\!\pr(\text{Poi}(z)=r-1)}
{\pr(\text{Poi}(z)\ge r)}.
\end{equation}

Since, by \eqref{E[Z]grows}, both $d E[Z_i]/dz_i$ and $d E[Z_o]/dz_o$ are positive, we have 
\begin{Corollary}\label{dzi/dt,dzo/dt}
\[
\frac{d z_i}{dt} = - \frac{E_o z_i}{L \mu}, \quad \frac{d z_o}{dt} = - \frac{E_i z_o}{L \mu}. 
\]
\end{Corollary}
\begin{proof} Apply $\tfrac{d\ex[Z_{i,o}]}{dt}=\tfrac{d\ex[Z_{i,o}]}{dz_{i,o}}\cdot \tfrac{dz_{i,o}}{dt}$.
\end{proof}

These surprisingly simple formulas yield that $\ex[Z_i]$ and $\ex[Z_o]$ both decrease as $t$ increases.
So $E_i(t)$, the total in-degree of the in-light vertices plus $\ex[Z_i]$ times the total
number of the in-light/out-heavy vertices, is $O(L(t))$, $L(t)$ being the total number of non-isolated
light vertices, uniformly for $t<T=\sup\{t:\,L(t)>0\}$;  $E_o(t)=O(L(t))$ as well. Since also
$v, v^i, v^o$ are $O(\mu)$, it follows then
that the RHS's of the differential equations in \eqref{system6} are bounded, {\it uniformly\/}
for $t<T$. In fact, since $v_{a,b}(t)=O(\mu)$ for $(a,b)\neq (0,0)$ as well, the RHS's of the
detailed differential equations \eqref{detsystem} are bounded as well. Using the definition of
$L(t)$ in \eqref{L(t)=}, we conclude that $|L'(t)|$ is bounded uniformly for $t<T$. Repeatedly 
differentiating both sides of the system \eqref{detsystem}, we conclude that all fixed order
derivatives of $v_{0,0}(t),\dots, \mu(t)$, whence of $L(t)$, are bounded for $t<T$.

These key qualitative results aside, Corollary \ref{dzi/dt,dzo/dt}  combined with the two bottom equations in \eqref{system6} also produces a crucial pair of {\it integrals\/}  of the dynamic system: 
\begin{Lemma}\label{twoconst}
The following two functions of $\bold s$ are constant along the trajectory $\bold s(t)$:
\begin{equation*}
\Phi_1(\bold s)=\frac{nz_i z_o}{\mu}, \qquad \Phi_2(\bold s)=\frac{np_{k_1}(z_i) p_{k_2}(z_o)}{v},
\end{equation*}
where $p_k(z):= \sum_{j \geq k} e^{-z} z^j/j! = Pr(\text{Poi}(z) \geq k)$. 
\end{Lemma}

\begin{proof}
For $\Phi_1(\bold s)$, note that
\begin{equation*}
\frac{1}{z_i} \frac{dz_i}{dt} + \frac{1}{z_o} \frac{dz_o}{dt} = -\frac{E_o}{L \mu} - \frac{E_i}{L \mu} = \frac{1}{\mu} \frac{d \mu}{dt}.
\end{equation*}
Consequently,
\begin{equation*}
\frac{d}{dt} \log \left( \frac{z_i z_o}{\mu} \right) = 0 \implies \frac{z_i z_o}{\mu} \equiv \text{constant}.
\end{equation*}
Turn to $\Phi_2(\bold s)$. We already used the identity
\[
\frac{d p_k(z)}{dz}=\frac{P(\text{Poi}(z)= k-1)}{\pr(\text{Poi}{z}\ge k)}=\frac{k_1P(Z(z)=k)}{z},
\]
where $Z(z)$ is $\text{Poi}(z)$, conditioned on ``$\text{Poi}(z)\ge k$''.  Applying it again, we have
\[
\frac{d\log p_{k_{1,2}}(z_{i,o})}{dt} = \frac{k_{1,2}P(Z_{i,o}=k_{1,2})}{z_{i,o}}\cdot
\left(-\frac{E_{o,i}z_{i,o}}{L\mu}\right)=-\frac{k_{1,2}P(Z_{i,o}=k_{1,2})E_{o,i}}{L\mu}.
\]
Therefore
\begin{equation*}
\frac{d}{dt}\bigl[ \log  p_{k_1}(z_i)  + \log  p_{k_2}(z_o) \bigr]= \frac{1}{v} \frac{dv}{dt} = \frac{d}{dt} \log v,
\end{equation*}
implying that $p_{k_1}(z_i) p_{k_2}(z_o)/v$ is constant. 
\end{proof}

\subsection{When does the trajectory terminates at a finite time?}  Corollary \ref{dzi/dt,dzo/dt} and Lemma \ref{twoconst} enable us to obtain a key criterion for finiteness of  $T\,(:=\sup\{t:\,L(t)>0\})$.
\begin{Lemma}\label{ziobound}
If $z_i(t),\,z_o(t)>0$ are bounded away from zero uniformly for $t<T$,  then 
\begin{align}
T\le\, \mu&(0)\log\frac{z_i(0)z_o(0)}{\inf_t (z_i(t)z_o(t))}<\infty,\label{T<oo}\\
L(T-)=\mu_i(T-)&=\mu_o(T-)=0, \qquad v(T-),\,\mu(T-)>0. \label{leftlimits}
\end{align}
\end{Lemma}
\begin{Remark} Unlike $v(t)$, $\mu(t)$, the functions $L(t)$, $\mu_i(t)$, $\mu_o(t)$ are not necessarily
monotone; so existence of their limits is a part of the claim. In the sequel, we will drop
``minus'' from $T-$, whenever $T<\infty$.
\end{Remark}
\begin{proof} By Corollary \ref{dzi/dt,dzo/dt},
\begin{equation}\label{d/dt(logs)}
\frac{d}{dt}\bigl(\log(1/z_i)+\log(1/z_o)\bigr)=\frac{E_i+E_o}{L}\frac{1}{\mu},\quad (t<T).
\end{equation}
Since $L(t)$ is the total number of non-isolated light vertices, we have 
\begin{align*}
E_i(t)+E_o(t)=& \sum_{a,b\atop (a,b)\neq (0,0)}\!\!\!\!(a+b)v_{a,b}+\sum_a (a+\ex[Z_o]) v_{a,\bullet}+\sum_b (b+\ex[Z_i])v_{\bullet,b}\\
\ge&\,\sum_{a,b\atop (a,b)\neq (0,0)}\!\!\!\!v_{a,b}+\sum_a v_{a,\bullet}+\sum_b v_{\bullet,b}
=L(t),
\end{align*}
because $\ex[Z_i]\ge k_1\ge 1$, $\ex[Z_o]\ge k_2\ge 1$. So, integrating the equation \eqref{d/dt(logs)},
 we obtain
\[
\log\frac{z_i(0)z_o(0)}{z_i(t)z_o(t)}\ge \int_0^t\frac{d\tau}{\mu(\tau)}.
\]
Using $\inf_t z_i(t)>0$, $\inf_t z_o(t)>0$ and $\mu(\tau)\le \mu(0)$, we conclude that
\[
t\le \mu(0)\log\frac{z_i(0)z_o(0)}{\inf_t (z_i(t)z_o(t))}\Longrightarrow T\le  \mu(0)\log\frac{z_i(0)z_o(0)}{\inf_t (z_i(t)z_o(t))}<\infty.
\]
So, by Lemma \ref{twoconst}, we have
\begin{align*}
v(T-)&=v(0)\frac{p_{k_1}(z_i(T-))p_{k_2}(z_o(T-))}{p_{k_1}(z_i(0))p_{k_2}(z_o(0))}>0,\\
\mu(T-)&=\mu(0)\frac{z_i(T-)z_o(T-)}{z_i(0)z_o(0)}>0.
\end{align*}
Finally, by definition of $T$, there exists a sequence $t_s\uparrow T$ such that $L(t_s)\to 0$.
Since $|L'(t)|=O(1)$ uniformly for $t<T$, we see then that $L(T-):=\lim_{t\uparrow T}L(t)$ exists, and it is $0$. So $\mu_i(T-)$, $\mu_o(T-)$ exist, and both are $0$.
\end{proof}
Next we will show that, subject to certain conditions on $(z_i(0),z_o(0),v(0),$ $\mu(0))$, 
the parameters $z_i(t)$, $z_o(t)$ are indeed bounded away from zero, whence the conclusion of Corollary 
\ref{ziobound}  holds. To state the result,
introduce
\[
F_k(x)=\frac{x}{p_k(x)p_{k-1}(x)},\quad (k\ge 0);
\]
so $F_0(x)=x$. Since $p_j(x)$ increases with $x$, we have that $F_k(x)/x$ decreases
with $x$. Introduce the notation $\bold r=(z_i,z_o,v,\mu)$.
\begin{Lemma}\label{1,2}
Suppose that, for some $x_1>0$, $x_2>0$,
at time $t=0$ we have 
\begin{align*}
&P(\bold r):=\frac{z_i}{p_{k_1}(z_i)}\cdot\frac{z_o}{p_{k_2}(z_o)} \cdot\frac{v}{\mu} > 
\sqrt{F_{k_1}(x_1)F_{k_2}(x_2)},\\
&\frac{z_i}{x_1},\,\frac{z_o}{x_2} \in \left(1,\frac{P^2(\bold r)}{F_{k_1}(x_1)F_{k_2}(x_2)}\right).
\end{align*}
Then $z_i(t)>x_1$, $z_o(t)>x_2$ for all $t<T$. Consequently, 
$T\le  \mu(0)\log\tfrac{z_i(0)z_o(0)}{x_1x_2}$,  and
 $v(T)>0$, $\mu(T)>0$.
  \end{Lemma}

\begin{proof}
First, note that 
\begin{equation}\label{fracfrac}
\begin{aligned}
P^2(\bold r) =&\, \frac{z_i}{p_{k_1}(z_i) p_{k_1-1}(z_i)} \cdot \frac{z_o}{p_{k_2}(z_o) p_{k_2-1}(z_o)}\\
&\times \frac{\mu-\mu_i}{v^i}\, \frac{\mu-\mu_o}{v^o} \,\frac{v^2}{\mu^2}\\
 & \leq  \frac{z_i}{p_{k_1}(z_i) p_{k_1-1}(z_i)}\cdot \frac{z_o}{p_{k_2}(z_o) p_{k_2-1}(z_o)}\\
 &=F_{k_1}(z_i)F_{k_2}(z_o).
\end{aligned}
\end{equation}
Suppose $\bold r(0)$ meets the conditions of Lemma \ref{1,2};
in particular,  $z_i(0)>x_1$, $z_o(0)>x_2$. We know that $z_i(t)$ and $z_o(t)$ are decreasing along the trajectory. Suppose that for some
$t<T$ either $z_i(t)\le x_1$ or $z_o(t)\le x_2$.  Let $t_0>0$ be the smallest such $t$.  Suppose, for instance, that $z_i(t_0)=x_1$ and $z_o(t_0) \geq x_2$. 
Using constancy of  $P(\bold r(t))$ and \eqref{fracfrac},  for $t=t_0$ we have
\begin{align*}
P^2(\bold r(0))\le&\, F_{k_1}(x_1)F_{k_2}(z_o(t_0))\\
\le&\,F_{k_1}(x_1)\frac{F_{k_2}(z_o(t_0))}{z_o(t_0)}\,z_o(t_0)\\
\le&\,\frac{F_{k_1}(x_1)F_{k_2}(x_2)}{x_2}\,z_o(0),
\end{align*}
as $F_{k_2}(x)/x$ decreases. Therefore
\begin{equation*}
z_o(0)\ge x_2\frac{P^2(\bold r(0))}{F_{k_1}(x_1)F_{k_2}(x_2)},
\end{equation*}
which contradicts the condition on $z_o$ in this lemma. Thus $z_i(t)>x_1$, $z_o(t)>x_2$
for all $t<T$. Using Lemma \ref{ziobound}  we complete the proof. 
\end{proof}

\subsection{Threshold edge density for termination at a finite time} For the likely parameters coming from $D(n,m=cn)$, we have that $z_i(0) \approx c$ and $z_o(0) \approx c$ and
\begin{equation*}
P(\bold r(0))=\frac{z_i(0)}{p_{k_1}(z_i(0))}\frac{z_o(0)}{p_{k_2}(z_o(0))} \frac{v(0)}{\mu(0)} \approx c.
\end{equation*}
Motivated by these observations, we focus then on the initial states of the deletion process such that
\begin{equation}\label{motiv}
\begin{aligned}
&(1-\eps)\max\{z_i(0),z_o(0)\}\le P(\bold r(0)) \le (1+\eps)\min\{z_i(0),z_o(0)\};\\
&\qquad\qquad\qquad 1-\eps\le \frac{v(0)}{np_{k_1}(z_i(0))\,p_{k_2}(z_o(0))} \le 1+\eps,\\
&(1-\eps)n\max\{z_i(0),z_o(0)\}\le\mu(0)\le (1+\eps)n \min\{z_i(0),z_o(0)\};
\end{aligned}
\end{equation}
here $\eps\in (0,1/2)$.  The tuple $\bold r(0)=\bold r_c:=(c,c,np_{k_1}(c)p_{k_2}(c), cn)$ is definitely admissible for every
given $c$ and $\eps\in (0,1)$. Eventually we will send $\eps$ to zero. 
\begin{Corollary}\label{g-large} There exists $\gamma_0=\gamma_0(k_1,k_2)>0$ such that if the condition \eqref{motiv} is met and $P(\bold r(0))>\gamma_0$, 
then the conditions \eqref{T<oo}, \eqref{leftlimits} hold, i.e. the process terminates at a finite time.
\end{Corollary}

\begin{proof} Pick $\gamma>0$. If $P(\bold r(0))\ge \gamma$, then by the right inequality in the
first line of \eqref{motiv},  $z_i(0),\,z_o(0)\ge
\tfrac{\gamma}{1+\eps}\ge \tfrac{\gamma}{2}$. Choose $x_1=x_2=1,$ say; then $z_i(0)>x_1$, $z_o(0)>x_2$ for $\gamma>2$.  Furthermore, the left inequality in the first line of \eqref{motiv}
and $\eps<1/2$ imply that
\[
\max\left\{\frac{z_i}{x_i},\frac{z_o}{x_2}\right\}=\max\{z_i,z_o\}\le 2P(\bold r(0))\le \frac{P^2(\bold r(0))}{F_{k_1}(x_1)F_{k_2}(x_2)},
\]
the last inequality holding provided that 
\[
P(\bold r(0))\ge 2 F_{k_1}(x_1)F_{k_2}(x_2)=2 F_{k_1}(1)F_{k_2}(1).
\]
So we can choose $\gamma_0(k_1,k_2)=\max\{3,2 F_{k_1}(1)F_{k_2}(1)\}$.
The claim then follows from Lemma \ref{1,2}.
\end{proof}
Corollary \ref{g-large} implies: if $c \ge c(k_1,k_2)$, then for the likely values of the initial state of $D(n,m=cn)$ the trajectory terminates at a finite time $T$ and the condition \eqref{leftlimits}
hold.
So we introduce 
\[
\gamma^*(\eps)=\inf\{\gamma>0:\,P(\bold r(0))\ge \gamma\text{ and }\eqref{motiv}
\Longrightarrow \eqref{T<oo}\text{ and }\eqref{leftlimits}\};
\]
by Corollary \ref{g-large}, we have $\gamma^*(\eps)\le \gamma_0(k_1,k_2)$.
\begin{Lemma}\label{g*=min} {\bf (i)\/} The infimum $\gamma^*(\eps)$ is positive, and more precisely $\gamma^*(\eps)\ge C(1-\eps)$, 
for some absolute constant $C=C(k_1,k_2)$. {\bf (ii)\/} For $\eps\le 1/2$, there exists $\bold r(0;\eps)$ such that $P(\bold r(0;\eps))=
\gamma^*(\eps)$ and $\bold r(t;\eps)$ terminates at a finite $T(\eps)=O(n)$ with $\mu(T(\eps);\eps)=\Theta(n)$, $v(T(\eps);\eps)=\Theta(n)$, with $O(n)$, $\Theta(n)$ being uniform.
\end{Lemma}
\begin{proof} For every $\gamma>0$  there exists $\bold r(0)$ satisfying \eqref{motiv}
with $P(\bold r(0))=\gamma$. Let $\gamma>\gamma^*(\eps)$. By the definition of $\gamma^*(\eps)$,  the process $\bold r(t)$ terminates at a finite time $T$,
with $L(T)=v_i(T)=v_o(T)=0$, and $\mu(T)>0$, $v(T)>0$. 

{\bf (i)\/} Suppose $k_1\ge 2$. By the definition of $P(\bold r)$, and constancy of $P(\bold r(t))$ for $t\in [0,T]$, we have
\begin{equation}\label{<P(boldr0)}
\begin{aligned}
\frac{z_i(T)}{\pr(\text{Poi}(z_i(T))\ge k_1)}\le&\,\frac{z_i(T)}{\pr(\text{Poi}(z_i(T))\ge k_1)\pr(\text{Poi}(z_o(T))\ge k_2-1)}\\
=&\,P(\bold r(T))\frac{\tfrac{\mu(T)}{v(T)}}{\tfrac{z_o(T)\pr(\text{Poi}(z_o(T))\ge k_2-1)}{\pr(\text{Poi}(z_o(T))\ge k_2)}}\\
=&\,P(\bold r(T))=P(\bold r(0)).
\end{aligned}
\end{equation}
Observe that for $z_i(T)\le 1$, say,
\[
\frac{z_i(T)}{p_{k_1}(z_i(T))}=\Theta\bigl(z_i(T)^{-k_1+1}\bigr).
\]
Therefore, by \eqref{<P(boldr0)},
\begin{equation}\label{zi(T)>}
z_i(T)\ge\min\left\{\!1,\Theta\!\left(\!P(\bold r(0))^{-\tfrac{1}{k_1-1}}\!\right)\!\right\}.
\end{equation}
On the other hand, by \eqref{motiv},
\[
z_i(T)\le z_i(0)\le (1-\eps)^{-1}P(\bold r(0)).
\]
Thus
\[
(1-\eps)^{-1}P(\bold r(0))\ge \min\left\{\!1,\Theta\!\left(\!P(\bold r(0))^{-\tfrac{1}{k_1-1}}\!\right)\!\right\}.
\]
It follows easily that, for some absolute constant $C$,
\[
\gamma=P(\bold r(0))\ge \min\left\{\!(1-\eps),\,\Theta\!\left(\!(1-\eps)^{\tfrac{k_1-1}{k_1}}\right)\!\right\}\Longrightarrow \gamma^*(\eps)\ge C(1-\eps).
\]
{\bf (ii)\/} Let $\eps\le 1/2$ and $\gamma\in (\gamma^*(\eps),2\gamma^*(\eps)$.  By \eqref{motiv}, $z_i(0),\,z_o(0)$ are bounded away from both zero and infinity. So $z_i(T)$, $z_o(T)$ are
bounded away from $\infty$.  By \eqref{zi(T)>},  $z_i(T)$ is bounded away from zero. Furthermore,  $z_o(T)$ is bounded away from zero, too. Just like the case $k_1\ge 2$, this claim holds if $k_2\ge 2$. Let $k_2 \leq 1$.
We have
\[
\frac{z_i(T)\!\pr(\text{Poi}(z_i(T))\ge k_1-1)}{\pr(\text{Poi}(z_i(T))\ge k_1)}=\frac{\mu(T)}{v(T)}=
\frac{z_o(T)\!\pr(\text{Poi}(z_o(T))\ge k_2-1)}{\pr(\text{Poi}(z_o(T))\ge k_2)}.
\]
The first fraction is $k_1$ at least; in fact it exceeds $k_1$ as $z_i(T)$ is bounded away from zero. However, 
the infimum of the third fraction is $1$ at most,  if  $z_o(T)$ is not bounded away from zero.
Contradiction! 

In addition, by \eqref{motiv}, we have $\mu(T)=\Theta(n)$ and then, by the equation above, $v(T)=\Theta(n)$, both uniformly for all $\gamma$ in question.
So, by Lemma \ref{ziobound} we obtain that, uniformly again,
\[
T\le \mu(0)\log \frac{z_i(0)z_o(0)}{z_i(T)z_o(T)}=O(n).
\]
This bound is not obvious, since it relates to the differential
equations, rather than to the random deletion process itself.  A standard, sequential compactness, argument shows then existence of the limiting trajectory starting at some admissible
$\bold r(0;\eps)$ with $P(\bold r(0;\eps))=\gamma^*(\eps)$, that terminates at time $T(\eps)=O(n)$, 
with the big-Oh estimate uniform for $\eps\le 1/2$, and $z_i(T(\eps);\eps)$, $z_o(T(\eps);\eps)$ 
each bounded away from $0$. 
\end{proof}
\bigskip
The function $\gamma^*(\eps)$ is increasing as $\eps$ is decreasing, since the range of admissible $\bold r(0)$, defined in \eqref{motiv}, is shrinking. Since $\gamma^*(\eps)\le \gamma_0$, there exists a finite 
\[
0<c^*=\lim_{\eps\to 0}\gamma^*(\eps)=\lim_{\eps\to 0}z_i(0;\eps)=\lim_{\eps\to 0}z_o(0;\eps)=
n^{-1}\lim_{\eps\to 0}\mu(0;\eps).
\]
We can assume existence of a sequence $\eps^s\to
0$, such that the corresponding trajectory $\bold  r^s(t)$, $t\in [0,T(\eps^s)]$, converges to some
$\bold r^*(t)$, $t\in [0,T^*]$, where $T^*=\lim T(\eps^s)=O(n)$, $v^*(T^*)=\Theta(n)$, $\mu^*(T^*)=\Theta(n)$, and
\[
\bold r^*(0)=(z_i^*(0),z_o^*(0), v^*(0),\mu^*(0)):=(c^*,c^*, np_{k_1}(c^*)p_{k_2}(c^*), c^*n)=:\bold r_{c^*}.
\]
Thus $\bold r^*(0)$ is determined up to $c^*$. 
\bigskip

The next step is to identify $c^*$ more explicitly. Here is a preliminary discussion.
By \eqref{<P(boldr0)}, the terminal pair $(z_i,z_o):=(z_i(T),z_o(T))$ satisfies the system of two
equations,

\begin{equation}\label{zi,zo=}
\begin{aligned}
z_i=&\,P(\bold r(0))\,p_{k_1}(z_i)p_{k_2-1}(z_o),\\
z_o=&\,P(\bold r(0))\,p_{k_1-1}(z_i)p_{k_2}(z_o),
\end{aligned}
\end{equation}
or equivalently
\begin{equation}\label{system}
\begin{aligned}
\log z_i-\log p_{k_1}(z_i)-\log p_{k_2-1}(z_o)=&\,\log P(\bold r(0)),\\
\log z_o-\log p_{k_1-1}(z_i)-\log p_{k_2}(z_o)=&\,\log P(\bold r(0)).
\end{aligned}
\end{equation}
We know that this system has a solution $(z_i,z_o)=(z_i^*(T^*),z_o^*(T^*))$
for $P(\bold r(0))=c^*$.
Let $J(z_i,z_o)$ denote the Jacobian for this system, i.e. the determinant of the $2\times 2$ matrix, whose $r$-th row is the (transposed) gradient of the $r$-th LHS expression, $1\le r\le 2$.
Using
\[
\frac{d}{dz}\log p_k(z)=z^{-1}\phi_k(z),\quad \phi_k(z):=\frac{z\pr(\text{Poi}(z)=k-1)}{\pr(\text{Poi}(z)\ge k)},
\]
($\phi_{-1}(z) = \phi_0(z):=0$), we have
\[
J(z_i,z_o)=\frac{1}{z_iz_o}\text{det}\begin{pmatrix}
1-\phi_{k_1}(z_i)&-\phi_{k_2-1}(z_o)\\
-\phi_{k_1-1}(z_i)&1-\phi_{k_2}(z_o)\end{pmatrix},
\]
or
\begin{equation}\label{J=?}
\begin{aligned}
&\quad J(z_i,z_o)=\frac{H(z_i,z_o)}{z_iz_o};\\
H(z_i,z_o):=&\,\bigl(1-\phi_{k_1}(z_i)\bigr)\bigl(1-\phi_{k_2}(z_o)\bigr)-\phi_{k_1-1}(z_i)\phi_{k_2-1}(z_o).
\end{aligned}
\end{equation}
\begin{Lemma}\label{J=0}
\begin{equation}\label{H*=0}
H\bigl(z_i^*(T^*),z_o^*(T^*)\bigr)=0.
\end{equation}
\end{Lemma}
\begin{proof}  Here is a naive attempt to prove this lemma. If $H\bigl(z_i^*(T^*),$
$z_o^*(T^*)\bigr)\neq 0$, then, by the implicit function theorem,  for $c$ sufficiently close to $c^*$
from {\it below\/} the system \eqref{system} has
a positive solution $z_i=z_i(c),$ $z_o=z_o(c)$. This ought to contradict the
minimality of $c^*$. However, to have a genuine contradiction we need to establish a stronger
fact. Namely that if 
$H\bigl(z_i^*(T^*),z_o^*(T^*)\bigr)\neq 0$, then for $c$ sufficiently close to $c^*$ from below
the trajectory for $P(\mathbf r(0)) =c$ also terminates at a {\it finite\/} time $T=T(c)$. 

To this end, let us show first that the condition $H\bigl(z_i^*(T^*),z_o^*(T^*)\bigr)\neq 0$ rules
out degenerate behavior of the trajectory for $P(\bold r(0))=c^*$ at $t$ close to  termination moment $T^*=T(c^*)$.

Consider, for instance, the case $k_1\ge 2$, $k_2\ge 2$. Instead of $\mathbf R=(v^i, v^o ,\mu_i,$
$\mu_o, v,\mu)$,
It is convenient to introduce $\mathbf{\mathcal R}=(\boldsymbol{\rho}, v,\mu)$, where the $4$-dimensional $\boldsymbol{\rho}:=(v^i-v,v^o-v,\mu_i,\mu_o)$. Let $v_i$ ($v_o$ resp.) be the total
number of in-light (out-light) vertices of a positive in-degree (out-degree resp.). Since
\[
v_i\le \mu_i\le (k_1-1)v_1,\quad v_o\le \mu_o\le (k_2-1)v_o,
\]
we see that for all $t\le T^*$,
\begin{equation}\label{>L>}
\sum_{j=1}^4 \rho_j\ge L\ge (\max(k_1,k_2)-1)^{-1}\sum_{j=1}^4\rho_j.
\end{equation}
The double inequality \eqref{>L>} implies that $\frac{\sum_j\rho_j(t)}{L(t)}$ is sandwiched between
$1$ and $\max(k_1,k_2)-1$. So for some sequence $t_i\to T^*$  there exists a finite
$\boldsymbol\xi=\lim_{t_i\to T^*}\boldsymbol{\rho}(t)/L(t)\ge \mathbf 0$, and $\boldsymbol\xi\neq \mathbf 0$.

The system \eqref{system6} can be
rewritten as 
\begin{equation}\label{rhosystem}
\begin{aligned}
\frac{d\mathbf{\mathcal R}}{dt}&=\frac{1}{L(t) }\mathbf F(\mathbf{\mathcal R(t)}), \quad (L(t)>0),\\
\mathbf F(\mathbf{\mathcal R})&:= A(\mathbf{\mathcal R})\boldsymbol{\rho} + 
\mathbf D(\mathbf{\mathcal R}),
\end{aligned}
\end{equation}
where $A(\mathbf{\mathcal R})$ is a $6\times 4$ matrix with $\mathbf{\mathcal R}$-dependent
entries, uniformly bounded for a given initial $\mathbf{\mathcal  R}(0)$,  and $D(\mathbf{\mathcal R})\in \Bbb R^6$ is a remainder term such that $\|D(\mathbf{\mathcal R})\|=
O\bigl(\|\boldsymbol\rho\|^2\bigr)$.  Let $B(\mathbf{\mathcal R})$
denote the $4\times 4$ submatrix of $B(\mathbf{\mathcal R})$, formed by the first $4$ rows. It
can be easily obtained that, with $I_4$ standing for the $4\times 4$ identity matrix,
\begin{equation}\label{matrix}
\begin{aligned}
B(\mathbf{\mathcal R})&=C(\mathbf{\mathcal R})- I_4,\\
C(\mathbf{\mathcal R})&=
\begin{pmatrix}k_2P(Z_o=k_2)&0&\frac{k_2P(Z_o=k_2)}{\ex[Z_o]}&0\\
0&k_1P(Z_i=k_1)&0&\frac{k_1P(Z_i=k_1)}{\ex[Z_i]}\\
0&(k_1)_2P(Z_i=k_1)&0&\frac{(k_1)_2P(Z_i=k_1)}{\ex[Z_i]}\\
(k_2)_2P(Z_o=k_2)&0&\frac{(k_2)_2P(Z_o=k_2)}{\ex[Z_0]}&0
\end{pmatrix}
\end{aligned}
\end{equation}
 Somewhat laborious computations show that
 \begin{equation}\label{detB(R)=}
\begin{aligned}
\text{det}\,B(\mathbf{\mathcal R})&=\bigl(k_1P(Z_i=k_1)-1\bigr)\bigl(k_2P(Z_o=k_2)-1\bigr)\\
&\quad-
\frac{(k_1)_2P(Z_i=k_1)}{\ex[Z_i]}\cdot \frac{(k_2)_2P(Z_o=k_2)}{\ex[Z_o]}\\
&=\bigl(\phi_{k_1}(z_i)-1\bigr)\bigl(\phi_{k_2}(z_o)-1\bigr)-\phi_{k_1-1}(z_i)\phi_{k_2-1}(z_o)\\
&= H(z_i,z_o). 
\end{aligned}
\end{equation}
Thus  Lemma \ref{J=0} asserts that the submatrix $B(\mathbf{\mathcal R(T^*)})$
is {\it singular\/}. 

For the proof by contradiction, suppose that $B(\mathbf{\mathcal R(T^*)})$ is non-singular.
Since $d\boldsymbol{\rho}/dt$ is bounded, without loss of generality there exists 
a partial $\lim_{t_i \to T^*} d\boldsymbol{\rho}/dt$, which can not have positive components.  So we obtain from 
\eqref{rhosystem} that $\boldsymbol\xi:=$ $\lim_{t_i\to T^*}$ $L^{-1}(t)\boldsymbol{\rho}(t)\neq\mathbf 0$ satisfies
\begin{equation}\label{B(T(c*))xi<0}
B(\mathbf{\mathcal R}(T^*))\boldsymbol\xi \le \bold 0\Longrightarrow C(\mathbf{\mathcal R}(T^*))\boldsymbol\xi\le\boldsymbol\xi.
\end{equation}
As the matrix $C(\mathbf{\mathcal R}(T^*))$ is non-negative, and indecomposable, we
see that $\boldsymbol\xi>\mathbf 0$.
Moreover, by Perron-Frobenius theorem (Gantmacher \cite{Gan}), the spectral
radius, i.e. the largest, necessarily positive, eigenvalue of $C(\mathbf{\mathcal R}(T^*))$, is at most $1$, hence strictly below $1$ because $\text{det}\,(C(\mathbf{\mathcal R}(T^*))-I_4)=\text{det}\,B(\mathbf{\mathcal R}(T^*))\neq 0$. 
Therefore there exist  $\boldsymbol{\eta}>\bold 0$ and $\gamma>0$ and such that, with $T$ standing for ``transpose'',
\begin{equation}\label{eta}
\boldsymbol{\eta}^TB(\mathbf{\mathcal R}(T^*))<-\gamma\,\boldsymbol{\eta}^T.
\end{equation}

We need to show that, for every $\eps>0$, there exists $\delta>0$ such that  for $P(\bold r(0))=
c\in (c^*-\delta,c^*)$  the trajectory $\mathbf{\mathcal R}(t)$ terminates no later than 
$T^*+\eps$. Otherwise there exist $\eps>0$ and a sequence $\delta_{\nu}\downarrow 0$ such
that for $P(\bold r(0))=c^*-\delta_{\nu}$ the corresponding trajectory $\mathbf{\mathcal R}^{(\nu)}(t)$ does not terminate until $T^*+\eps$. By continuous dependence on the initial condition,
$\|\mathbf{\mathcal R}^{(\nu)}(T^*)-\mathbf{\mathcal R}(T^*)\|\to 0$. In particular,
$L^{(\nu)}(T^*)\to 0$, $\boldsymbol{\rho}^{(\nu)}(T^*)\to 0$, and $\boldsymbol{\eta}^T
B(\mathbf{\mathcal R}^{(\nu)}(T^*))\le -(\gamma/2)\,\boldsymbol{\eta}$.  By 
(equi)continuity of $\mathbf{\mathcal R}^{(\nu)}(t)$, there exists $\eps_0<\eps$ such that 
\[
\boldsymbol{\eta}^TB(\mathbf{\mathcal R}^{(\nu)}(t))\le -(\gamma/3)\, \boldsymbol{\eta},
\]
for all $t\in [T^*,T^*+\eps_0]$. From \eqref{rhosystem}
\[
\frac{d\boldsymbol{\rho}^{(\nu)}}{dt}=\frac{1}{L^{(\nu)}}\bigl[B(\mathbf{\mathcal R}^{(\nu)})
\boldsymbol{\rho}^{(\nu)}+\mathbf O(\|\boldsymbol{\rho}^{(\nu)}\|^2)\bigr].
\]
Multiplying this equation by $\boldsymbol\eta^T$, and using \eqref{>L>},  we have
\begin{align*}
\frac{d (\boldsymbol{\eta}^T\boldsymbol{\rho}^{(\nu)})}{dt}&\le -\frac{\gamma}{3}
\frac{(\boldsymbol{\eta}^T\boldsymbol{\rho}^{(\nu)})}{L^{\nu}}+O\bigl(\|\boldsymbol{\rho}^{(\nu)}\|
\bigr)\\
&=-\frac{\gamma}{3}
\frac{(\boldsymbol{\eta}^T\boldsymbol{\rho}^{(\nu)})}{L^{\nu}}+O\bigl(\boldsymbol{\eta}^T\boldsymbol{\rho}^{(\nu)}\bigr)\\
&\le -a +b\,(\boldsymbol{\eta}^T\boldsymbol{\rho}^{(\nu)}),
\end{align*}
where $a$ and $b$ are constants independent of $\nu$. Integrating this differential inequality
we obtain that  that $\boldsymbol{\eta}^T\boldsymbol{\rho}^{(\nu)}(t)$ cannot be positive
for 
\begin{align*}
t &\ge T^*+b^{-1}\log \frac{1}{1-\tfrac{b}{a}\bigl(\boldsymbol{\eta}^T\boldsymbol{\rho}^{(\nu)}(T^*)\bigr)}\\
&=T^* + O\bigl(\boldsymbol{\eta}^T\boldsymbol{\rho}^{(\nu)}(T^*)\bigr)=T^*+o(1),
\end{align*}
as $\nu\to\infty$. Contradiction! The proof of Lemma \ref{J=0} is complete.
\end{proof}

For $c\ge c^*$, the trajectory terminates at a finite time $T(c)$. By continuous 
dependence of the trajectory on the starting point, $T(c+)\ge T(c)$. Further, 
for a partial limit $\xi:=\lim_{t_i\to T(c)}L^{-1}(t)\boldsymbol\rho(t)\neq \bold 0$, we have an
extension of \eqref{B(T(c*))xi<0}, namely
\begin{equation}\label{B(T(c))xi<0}
B(\mathbf{\mathcal R}(T(c)))\boldsymbol\xi \le \bold 0\Longrightarrow C(\mathbf{\mathcal R}(T^*))\boldsymbol\xi\le\boldsymbol\xi.
\end{equation}
As in the proof above, it follows that $\boldsymbol\xi>\bold 0$. So, by \eqref{detB(R)=} and
the first two inequalities in \eqref{B(T(c))xi<0}, we have: for $z_i=z_i(T(c))$, $z_o=z_o(T(c))$,
\begin{align}
H(z_i,z_o):=(1-\phi_{k_1}&(z_i))(1-\phi_{k_2}(z_o))-\phi_{k_1-1}(z_i)\phi_{k_2-1}(z_o)\ge 0,
\label{H(T(c))<=0}\\
&\phi_{k_1}(z_i)<1,\quad \phi_{k_2}(z_o)<1.\label{<1,<1}
\end{align}

\subsection{Variational characterization of  $c^*$}
 \bigskip
{\bf (a)\/} Suppose $k_2 \leq 1$. Then we have $\phi_{k_2-1}(z^*_o)=0$, and
\[
\phi_{k_2}(z_o^*)-1 \leq \frac{z_o^* e^{-z_o^*}}{1-e^{-z_o^*}}-1=\frac{z_o^*}{e^{z_o^*}-1}-1 <0.
\]
So $z_i^*$ is the unique root of
\begin{equation}\label{zi*root}
\phi_{k_1}(z)=1\Longleftrightarrow\frac{1}{\sum_{j\ge k_1}\frac{z^{j-k_1}}{j!}}=(k_1-1)!.
\end{equation}
Consequently, from the first equation in \eqref{zi,zo=},
\begin{equation}\label{c*=}
c^*=\frac{z_i^*}{p_{k_1}(z_i^*)}=\frac{z_i^*}{\pr(\text{Poi}(z_i^*)\ge k_1)}.
\end{equation}
Further, given $z_i^*$, the parameter $z_o^*$ is the unique root of
\[
\frac{z}{\pr(\text{Poi}(z) \geq k_2)}=\frac{z_i^*\pr(\text{Poi}(z_i^*)\ge k_1-1)}{\pr(\text{Poi}(z_i^*)\ge k_1)}.
\]
Notice that the equation \eqref{zi*root} means that $z_i^*$ is the absolute minimum point
of $\tfrac{z}{p_{k_1}(z)}$.  
Indeed,
\[
\left(\frac{z}{p_k(z)}\right)'=\frac{1-\phi_k(z)}{p_k(z)},
\]
and $\phi_k(z)$ decreases with $z$, with $\phi_k(0)=k$, $\phi(\infty)=0$.
So $c^*$ is also the critical value for the birth of a giant 
$(k_1+1)$-core for the {\it undirected\/} random graph $G(n,m=cn)$ (see \cite{PitSpeWor} for the {\it undirected} core phase transition).\\

{\bf (b)\/} Suppose $k_1=k_2:=k\ge 2$. In this case $z_i^*=z_o^*=:z^*$. So the equation \eqref{J=?} and Lemma \ref{J=0} imply that
\begin{equation}\label{k,k-1}
1-\phi_k(z^*)-\phi_{k-1}(z^*)=0,
\end{equation}
or equivalently
\[
\left.\frac{d}{dz}\!\left(\log z -\log p_k(z)-\log p_{k-1}(z)\right)\right|_{z=z^*}=0.
\]
So, analogously to the case $k_2=1$,  $z^*$ is an absolute minimum  point of $F_k(z):=\tfrac{z}{p_k(z)p_{k-1}(z)}$.\\

 {\bf (c)\/} Generally,
\begin{Lemma}\label{c*,genk1,k2}  Suppose $\max\{k_1, k_2\} \geq 2$. Then 
\begin{equation}\label{c*=,genk1,k2}
\begin{aligned}
c^*&=\min_{z_i,z_o}\Psi(z_i,z_o),\\
\Psi(z_i,z_o)&:=\max\left\{\frac{z_i}{p_{k_1}(z_i)p_{k_2-1}(z_o)};\,\frac{z_o}{p_{k_2}(z_o)p_{k_1-1}(z_i)}\right\}.
\end{aligned}
\end{equation}
\end{Lemma}
\begin{proof}  The case $\min\{k_1,k_2\} \leq 1$ was effectively covered in the item {\bf (a)\/}.
So let us assume that $k_1>1$ and $k_2>1$. Since the function $\Psi(z_i,z_o)\to\infty$ for $\min(z_i,z_o)\to 0$ or $\max(z_i,z_0)$ $\to\infty$, it attains the infimum at an interior point $(\hat z_i,\hat z_o\}$. Since $\tfrac{z_i}{p_{k_1}(z_i)p_{k_2-1}(z_o)}$ and $\tfrac{z_o}{p_{k_2}(z_o)p_{k_1-1}(z_i)}$ are strictly decreasing as functions of
$z_o$ and $z_i$ respectively, we have
\[
\frac{\hat z_i}{p_{k_1}(\hat z_i)p_{k_2-1}(\hat z_o)}=\frac{\hat z_o}{p_{k_2}(\hat z_o)p_{k_1-1}(\hat z_i)}.
\]
Equivalently, for $Z_i$ and $Z_o$ being conditioned on the events $\{\text{Poi}(z_i) \geq k_1\}$ and $\{\text{Poi}(z_o) \geq k_2\}$, respectively, we have that
$\ex[Z_i]=\ex[Z_o]$ at $z_i=\hat z_i$, $z_o=\hat z_o$. The latter condition definitely holds for $z_i=z_i^*$,
$z_o=z_o^*$. Now  
\begin{equation}\label{d/dzE}
\frac{d}{dz_i}\ex[Z_i]=\frac{\text{Var}(Z_i)}{z_i}>0,\quad \frac{d}{dz_o}\ex[Z_o]=
\frac{\text{Var}(Z_o)}{z_o}>0.
\end{equation}
So the condition $\ex[Z_i]=\ex[Z_o]$ determines, implicitly, $z_o=F(z_i)$, where $F(z_i)$ is
continuously differentiable, and $F'(z_i)>0$.  By implicit differentiation,
\begin{equation}\label{F'(zi)=}
F'(z_i)=\frac{z_o}{z_i}\,\frac{1+\phi_{k_1-1}(z_i)-\phi_{k_1}(z_i)}{1+\phi_{k_2-1}(z_o)-\phi_{k_2}(z_o)},
\end{equation}
with the numerator and the denominator being both {\it positive\/}, see \eqref{Var(Z)}. Then 
\begin{multline}\label{J/(1+)}
z_i\frac{d}{dz_i}\left(\log \frac{z_i}{p_{k_1}(z_i)p_{k_2-1}(z_o)}\right)=1-\phi_{k_1}(z_i)-
\frac{z_i}{z_o}\phi_{k_2-1}(z_o)F^\prime(z_i)\\
=1-\phi_{k_1}(z_i)-\phi_{k_2-1}(z_o)\,\frac{1+\phi_{k_1-1}(z_i)-\phi_{k_1}(z_i)}{1+\phi_{k_2-1}(z_o)-\phi_{k_2}(z_o)}\\
=\frac{(1-\phi_{k_1}(z_i))(1-\phi_{k_2}(z_o))-\phi_{k_1-1}(z_i)\phi_{k_2-1}(z_o)}
{1+\phi_{k_2-1}(z_o)-\phi_{k_2}(z_o)}\\
=\frac{H(z_i,z_o)}{1+\phi_{k_2-1}(z_o)-\phi_{k_2}(z_o)}.
\end{multline}
Since
\[
\phi_{k_1}(z_i)-1<\phi_{k_1-1}(z_i),\quad \phi_{k_2}(z_o)-1<\phi_{k_2-1}(z_o),
\]
the function $H(z_i,z_o)$ is negative when at least one of $1-\phi_{k_1}(z_i)$, $1-\phi_{k_2}(z_o)$ is non-positive. Therefore $\hat z_i$, $\hat z_o$ satisfy
\begin{equation}\label{sat}
1-\phi_{k_1}(z_i)>0,\quad 1-\phi_{k_2}(z_o)>0.
\end{equation}
By \eqref{<1,<1} we know that the last inequalities hold for $z_i^*$, $z_o^*$. 
Since each $\phi_j(z)$ decreases with $z$, the function $H(z_i,z_o)$ increases
with $z_i$ as long as $(z_i,z_o=F(z_i))$ satisfy \eqref{sat}. So the condition $H(z_i^*,z_o^*)=0$ means that $(z_i^*,z_o^*)$ is the unique absolute
minimum point of $\tfrac{z_i}{p_{k_1}(z_i)p_{k_2-1}(z_o)}$, subject to constraint $\ex[Z_i]=
\ex[Z_o]$.
Hence $\hat z_i=z_i^*$, $\hat z_o=z_o^*$.
\end{proof}

For $c\ge c^*$, the terminal $z_i=z_i(T(c))$, $z_o=z_o(T(c))$ satisfy
\begin{equation}\label{two}
c=\frac{z_i}{p_{k_1}(z_i)p_{k_2-1}(z_o)}=\frac{z_o}{p_{k_1-1}(z_i)p_{k_2}(z_o)},
\end{equation}
So $(z_i,z_o)\neq (z_i^*,z_o^*)$ for $c>c^*$; consequently $H(z_i,z_o)\neq 0$, and whence
$H(z_i,z_o)>0$ by \eqref{H(T(c))<=0}. Arguing as in the proof of Lemma \ref{J=0} we obtain then that $T(c+)\le T(c)$, implying that $T(c+)=T(c)$.\\
Computationally, 
given $c\ge c^*$, $z_{i,o}(c)$ can be obtained as the limit of the monotone decreasing
sequence $z_{i,o}^{(r)}$
defined recursively as follows:  $z_i^{(0)}=z_o^{(0)}=c$, and for $r>0$, 
\begin{equation}\label{recur}
\begin{aligned}
z_i^{(r)}&=cp_{k_1}(z_i^{(r-1)})\,p_{k_2-1}(z_o^{(r-1)}),\\
z_o^{(r)}&=cp_{k_2}(z_o^{(r-1)})\,p_{k_1-1}(z_i^{(r-1)}).
\end{aligned}
\end{equation}

The interested reader may notice a conceptual similarity of the equations \eqref{recur} to the equations (2.1) in \cite{PitSpeWor}, suggested there as a heuristic, branching process related, explanation of the formula for the threshold value fo the edge density for birth of the $k$-core in the undirected case. Later Riordan \cite{Rio} found a proof of the threshold based on this approach. While it is not difficult to ``explain" the equations \eqref{recur}, just like those equations in \cite{PitSpeWor}, we decided to stick with our deletion process. Local nature of the one-step transition enabled us to describe, succinctly, the process as a Markov chain, with tractable expected state changes at each step. With the ODE system as a possible deterministic approximation, we were led to identification of the critical $c^*$, as a candidate for the $(k_1, k_2)$-core threshold. In the next, final, section we will prove the approximation property, thereby rigourously proving that $c^*$ is the threshold edge density for existence of a giant $(k_1, k_2)$-core in the graph $D(n,m)$.

\section{$c^*$ is the threshold edge density for a $(k_1k_2)$-core in $D(n,m)$}
For the starting digraph $D(n,m=[cn])$, w.h.p. the initial state $\bold s(0)$ is such that,
with $\delta_n:=n^{-1/2}\log n$,
\begin{equation}\label{s(0)=}
\begin{aligned}
v_{a,b}(0)&=(1+O(\delta_n)) n\pr(Z_i(c)=a)\pr(Z_o(c)=b),\quad a<k_1,\,b<k_2,\\
v_{a,\bullet}(0)&=(1+O(\delta_n)) n\pr(Z_i(c)=a)\pr(Z_o(c)\ge k_2),\quad a<k_1,\\
v_{\bullet,b}(0)&=(1+O(\delta_n)) n\pr(Z_i(c)\ge k_1)\pr(Z_o(c)=b),\quad b<k_2,\\
v(0)&=(1+O(\delta_n)) n\pr(Z_i(c)\ge k_1)\pr(Z_o(c)\ge k_2),\\
\mu(0)&=m=[cn].
\end{aligned}
\end{equation}
The random deletion process $\{\bold s(t)\}$ stops at 
\[
\tau:=\min\{t:\,\mu_i(t)=\mu_o(t)=0\}=\min\{t:\,L(t)=0\}. 
\]
\subsection{Supercritical case} 
\begin{Theorem}\label{super} If $c>c^*$ then w.h.p. the deletion
process delivers a $(k_1,k_2)$-core of size asymptotic to $np_{k_1}(z_i)p_{k_2}(z_0)$,
with about $n\tfrac{z_iz_o}{c}$ edges, and the parameters $z_i$, $z_0$ determined as the
limit of the recurrence \eqref{recur}.
\end{Theorem}
\smallskip

\begin{proof}
{\bf (a)\/} Consider the ODE trajectory $\bar {\bold s}_c(t)$,
i.e. the solution of \eqref{detsystem}, 
that starts at $\bold s(0)=\bold s_c(0)$ given by \eqref{s(0)=}, with the factor $1+O(\delta_n)$ dropped and $\mu(0)=cn$. We know that, for $c>c^*$,  the trajectory terminates at  a finite time 
$T(\bold s_c(0)):=T(c)=\Theta(n)$, and that 
\begin{equation}\label{H>0}
H\bigl[\bar z_i(T(\bold s_c(0))),\bar z_o(T(\bold s_c(0)))\bigr]>0.
\end{equation}
Since the RHS of  \eqref{detsystem} is a zero-degree homogeneous vector-function of $\bold s$,
independent of $t$, we can scale both the state $\bold s$ and the time $t$ by $n$. Using \eqref{H>0}
and an argument similar to the proof of Lemma \ref{J=0}, we obtain that for every $\bold s(0)$
satisfying \eqref{s(0)=}, the ODE trajectory terminates at a finite time $T(\bold s(0))$ such
that  $|T(\bold s(0))-T(\bold s_c(0))|=O(n^{1/2}\log n)$. In particular, it follows from \eqref{H>0}
that 
\begin{equation}\label{infH>0}
\inf\bigl\{H\bigl[\bar z_i(T(\bold s(0))),\bar z_o(T(\bold s(0)))\bigr]:\,\bold s(0)\text{ meets }\eqref{s(0)=}
\bigr\}>0.
\end{equation}
It is convenient to extend the definition of 
$\bar {\bold s}(t)$, setting $\bar {\bold s}(t)\equiv \bar {\bold s}(T(\bold s(0)))$ for
$t\ge T(\bold s(0))$. 

{\bf (b)\/} Introduce $\bold S_{\omega}$, the set of all $\bold s$ meeting the constraints
\begin{equation}\label{Somega}
\mu-\mu_{i,o} -k_{1,2}v^{i,o}\ge \omega,\quad v\ge \frac{n}{\log n}.
\end{equation}
where $L=L(\bold s)$ is the total number of non-isolated light vertices, and $\omega=\omega(n)\to\infty$ however slowly. Every $\bold s(0)$, satisfying \eqref{s(0)=}, certainly belongs to
$\bold S_{\omega}$. For the deletion process $\bold s(t)$, let $\tau_{\omega}$ be the first $t\le\tau$ such that $\bold s(t)\notin \bold S_{\omega}$. Denoting 
the (explicit part of) RHS of \eqref{edvab}-\eqref{Edmu} by $\bold f(\bold s)$, we have: for 
$t<\tau_{\omega}$,
\begin{equation}\label{trend}
\ex[\bold s(t+1) - \bold s(t)\,|\bold s(t)]=\bold f(\bold s(t))+O(n^{-1}\log^6n).
\end{equation}
In addition, by \eqref{X,Ysmall},
\begin{equation}\label{smallchange}
\pr(\|\bold s(t+1)-\bold s(t)\|\le \log^2 n\,\bigr|\, \bold s(t)) \ge 1 -e^{-\log^2 n}.
\end{equation}
The attendant ODE system \eqref{detsystem} is 
\[
\frac{d\bar{\bold s}(t)}{dt}=\bold f(\bar{\bold s}(t)),\quad t<T(\bold s(0)).
\]
By a general purpose theorem due to Wormald, \cite{Wor1}, with probability
\[
1-O\left(ne^{-0.5(\log^2 n)\log\log n}+n^{1/3}e^{-\log^2 n}\right)\ge 1- e^{-0.5\log^2n},
\]
we have:  
\[
\max\bigl[\|\bold s(t) -\bar{\bold s}(t)\|\,:\, t\le \tau_{\omega}(\bold s(0))\bigr]
 =O(\log^6 n).
\]
 where $\tau_{\omega}(\bold s(0)):=\min\{T_0,\tau_{\omega}\}$, and $T_0:=
 \lceil T(\bold s(0))\rceil$.
 (The conditions \eqref{trend}, \eqref{smallchange} are particular examples of  ``Trend hypothesis'' and ``Boundedness hypothesis'' in \cite{Wor1}.) 
 
 On the ODE trajectory $\bar z_{i,o}(\tau_{\omega}(\bold s(0)))\ge \bar z_{i,o}(T(\bold s(0))$,  because $\bar z_{i,o}(t)$ are decreasing with $t$.  Likewise
 $\bar \mu(\tau_{\omega}(\bold s(0)))$ and $\bar v(\tau_{\omega}(\bold s(0)))$ are both of order $n$. From the definition
 of $z_{i,o}$, it follows then that $\bar\mu -\bar\mu_{i,o} -k_{1.2}\bar{v}^{\,i,0}$ are both of order $n$ at $\tau_{\omega}(\bold s(0))$. So if $\tau_{\omega}<T_0$ then 
necessarily $L(\tau_{\omega})=0$. If, on the other hand,  $\tau_{\omega}\ge T_0$, then 
 $\bar L(\tau_{\omega}(\bold s(0))=\bar L(T_0)=0$. 
Therefore with probability $\ge 1- e^{-0.5\log^2n}$ we either end up (1)
with a giant $(k_1,k_2)$-core at $\tau_{\omega}$ or (2) with $O(\log^6 n)$ light vertices at $T_0$.

Consider the second alternative. Let $\tau'$ be the first $t> T_0$ such that either $L(t)=0$ or
 $\|\bold s(t) -\bar{\bold s}(t)\|\ge \log^9 n$; obviously $\tau'<\infty$.  In light of
\eqref{smallchange},  we have
\[
\|\bold s(\tau') -\bar{\bold s}(\tau')\|\le  2\log^9 n,\quad \forall\,t\le \tau',
\]
except on the event of probability $< e^{-\log ^2 n}$.
Furthermore,  again by \eqref{smallchange} and 
$\|\bold s(\tau_{\omega})-\bar{\bold s}(T_0)\|=O(\log^6 n)$,  we
have
\begin{equation}\label{diff}
\pr\bigl(\{L(\tau')\neq 0\}\setminus \{\tau'-T_0 =\Omega(\log^7 n)\}\,|\,\bold s(T_0)\bigr)\le e^{-\log^2 n}.
\end{equation}
What follows next is a probabilistic, discrete-time, counterpart of our proof of Lemma 3.1. 
Let $t\in [T_0,\tau']$, so that $\bar{\bold s}(t)=\bar{\bold s}(T(\bold s(0)))$. Using the notations $\bold{\mathcal R}=(\boldsymbol\rho,v,\mu)$, where $\boldsymbol\rho=(v^i-v,v^o-v,\mu_i,\mu_0)$, with a little work we obtain from \eqref{system6}:
\begin{equation}\label{rho(t+1)-rho(t)}
\ex[\boldsymbol\rho(t+1)-\boldsymbol\rho(t)\,|\bold s(t)]=\frac{1}{L(t)}B
\boldsymbol\rho(t)+\bold O\bigl(n^{-1}\log^{18}n\bigr).
\end{equation}
The $4\times 4$ matrix $B=C(\bold{\mathcal R})-I_4$ is defined in \eqref{matrix}, and $\bold{\mathcal R}=\bold{\mathcal R}(T(\bold s(0)))$, which is the terminal value of $(v^{i,o},\mu_{i,o},v,\mu)$
for the ODE trajectory that starts at $\bold s(0)$.  

The equation \eqref{infH>0} implies existence of $\boldsymbol\eta>\bold 0$, and $\gamma>0$ such that, analogously to \eqref{eta},  we have $\boldsymbol\eta^T B\le -\gamma\eta^T$, ($T$
standing for  ``transpose''),  for every $\bold s(0)$ meeting the constraints \eqref{s(0)=}. Let $r(t)=\boldsymbol\eta^T \boldsymbol\rho(t)$. By the left inequality in \eqref{>L>},  it follows from \eqref{rho(t+1)-rho(t)} that
\[
\ex[r(t+1)-r(t)\,|\bold s(t)]\le -\frac{\gamma}{2}\cdot \frac{r(t)}{L(t)} +O\bigl(n^{-1}\log^{18}n\bigr)\le
-\sigma,
\]
where $\sigma>0$ is fixed. We emphasize that this bound holds for all $t\in [T_0,\tau']$.
 
The rest is short. By \eqref{smallchange},  it follows easily that for a fixed, sufficiently small, $u>0$,
and $t\in [t_0,\tau'-1]$
\[
\ex\bigl[\exp\bigl(u(r(t+1)-r(T_0)\bigr)\,|\bold s(t)\bigr]
\le e^{-u\sigma/2}\exp\bigl(u(r(t)- r(T_0)\bigr).
\]
Therefore 
\[
M(t):=\exp\bigl(u(t-T_0)\sigma/2\bigr)\,\exp\bigl(u(r(t)- r(T_0))\bigr),\quad t\in [T_0,\tau'],
\]
is a super-martingale, whence  (Durrett \cite{Dur}, Section 4.7)
\[
\ex[M(\tau')\,|\,\bold s(T_0)]\le M(T_0)=1.
\]
Now, on the event $\{L(\tau')\neq 0\}$ we have $\tau' -T_0=\Omega(\log^7 n)$. So from
\begin{multline*}
1\ge \ex[M(\tau')\,|\,\bold s(T_0)]\ge e^{-ur(T_0)}\ex\bigl[\exp\bigl(u(t-T_0)\sigma/2\bigr)\,|\,\bigr]\\
=\exp\bigl(-O(\log^6 n)+\Omega(\log^7 n)\bigr)\!\!\pr\bigl(\tau'-T_0 =\Omega(\log^7 n)\,|\,\bold s(T_0)\bigr)
\end{multline*}
we get
\[
\pr\bigl(\tau'-T_0 =\Omega(\log^7 n)\,|\,\bold s(T_0)\bigr)\le \exp\bigl(-\Omega(\log^7 n)\bigr),
\]
implying, via \eqref{diff}, that 
\[
\pr\bigl(L(\tau')\neq 0\,|\,\bold s(T_0)\bigr)\le e^{-0.5\log^2 n}. 
\]
\end{proof}

\subsection{Subcritical case}
\begin{Theorem}\label{sub} If $c<c^*$ then w.h.p. the $(k_1,k_2)$-core is empty.
\end{Theorem}
\begin{proof} This time for the ODE solution with $\bold s(0)$ meeting constraints \eqref{s(0)=}
we have that $\min\{z_i(t),z_o(t)\}\to 0$. Since the ODE system \eqref{detsystem} has two
explicit integrals, 
\[
\Phi_1(\bold s)=\frac{ z_iz_o}{\mu/n},\quad \Phi_2(\bold s)=
\frac{p_{k_1}( z_i)p_{k_2}(z_o)}{v/n},
\]
we obtain that $\mu(t) \to 0$, whence $v(t) \to 0$, as $t\to\infty$. Thus, given any $b<p_{k_1}(c)p_{k_2}(c)$, there exists
a time $T_b\approx a(b)n$ such that $v(T_b)=bn$. Applying Wormald's theorem we see that w.h.p. in the random deletion process after about $T_b$ steps the total number of the remaining doubly heavy
vertices is sharply concentrated around $bn$ vertices. Since $b>0$ is arbitrary, we obtain that w.h.p. the number of vertices in the $(k_1, k_2)$-core is below $0.8 \alpha(k,c) n$, see \eqref{akc bound} for $\alpha(k,c)$. By the directed counterpart of \L uczak's result \cite{Luc1} for $G(n,p=c/n)$, it follows that w.h.p. the $(k_1, k_2)$-core of $D(n, m=[cn])$ is empty.
\end{proof}

\end{document}